\documentclass[oneside, reqno]{amsart}
\usepackage[utf8]{inputenc}
\usepackage{amsmath,amssymb,amsthm, graphicx,float,url}
\usepackage[dvipsnames]{xcolor}
\usepackage{tikz-cd} 
\usepackage{thm-restate}
\usepackage{physics} 
\usepackage[capitalize]{cleveref}
\crefname{subsection}{subsection}{subsections}
\usepackage[symbol]{footmisc}
\usepackage{nicefrac} 
\allowdisplaybreaks 

\usepackage{lineno}

\newcommand{\constMap}{1_N}

\newcommand{\PRESEQNA}{\ensuremath{\mathsf{PresEqnA}}} 
\newcommand{\PRESEQNB}{\ensuremath{\mathsf{PresEqnB}}} 
\newcommand{\MATEQN}{\ensuremath{\mathsf{MatEqn}}} 
\newcommand{\EQNMAT}{\ensuremath{\mathsf{EqnMat}}} 
\newcommand{\equationsProbNEW}{\textsc{EquationsSubspan}\xspace}
\newcommand{\MatrixSubspanA}{\textsc{MatrixSubspanA}\xspace}
\newcommand{\MatrixSubspanB}{\textsc{MatrixSubspanB}\xspace}

\newcommand{\CNF}{\ensuremath{\mathsf{CNF}}}

\numberwithin{equation}{section}

\usepackage[a4paper]{geometry}
\usepackage{mathtools}
\usepackage{stackrel}

\newcommand{\bi}{\begin{itemize}}
	\newcommand{\ei}{\end{itemize}}
\newcommand{\be}{\begin{enumerate}}
	\newcommand{\ee}{\end{enumerate}}

\usepackage[english]{babel}
\usepackage{amsthm}
\usepackage{amsmath,amssymb,tabularx,gastex,hhline,rotating,enumerate,xspace}

\usepackage{tikz}
\usetikzlibrary{calc,decorations.pathmorphing,patterns,arrows,decorations.markings,positioning}
\usetikzlibrary{automata,chains,fit,shapes}

\renewcommand{\geq}{\geqslant} \renewcommand{\leq}{\leqslant}

\newenvironment{aw}{\noindent\color{blue} AW :  }{}

\newtheorem{theorem}{Theorem}[section]
\newtheorem{proposition}[theorem]{Proposition}
\newtheorem{corollary}[theorem]{Corollary}
\newtheorem{lemma}[theorem]{Lemma}
\newtheorem{conjecture}[theorem]{Conjecture}

\theoremstyle{definition}
\newtheorem{definition}[theorem]{Definition}
\newtheorem{remark}[theorem]{Remark}
\newtheorem{example}[theorem]{Example}
\newtheorem{theoremx}{Theorem}

\newcommand{\sgn}{\textnormal{sgn}}
\newcommand{\Aut}{\textnormal{Aut}}

\newcommand{\tildef}{\tilde{f}}

\renewcommand{\geq}{\geqslant} \renewcommand{\leq}{\leqslant}  

\newcommand{\PMonlyExt}{\mathrm{SpecialExt}}

\newcommand{\Z}{\mathbb Z}
\newcommand{\N}{\mathbb N}
\newcommand{\R}{\mathbb R}
\newcommand{\bbX}{\mathbb{X}}
\newcommand{\bbY}{\mathbb{Y}}
\newcommand{\Fin}{\mathrm{Fin}}
\newcommand{\Arb}{\mathrm{FinPres}}
\newcommand{\DPclass}{\mathrm{Abel}\times\mathrm{Fin}}

\newcommand{\FreeAb}{\mathrm{FreeAbel}}
\newcommand{\PMonlySemi}{\ensuremath{\mathrm{RestrAbelSemi}}}
\newcommand{\vC}{\mathrm{VirtCyclic}}

\newcommand{\spanZ}[1]{\operatorname{span}({#1})}
\newcommand{\spanZp}[1]{\operatorname{span_{\Z_p}}({#1})} 
\newcommand{\spanb}[1]{\operatorname{span}_{b}({#1})}
\newcommand{\spanX}[2]{\operatorname{span}_{#2}({#1})}

\newcommand{\Epi}[2]{\operatorname{Epi}(#1,#2)}

\newcommand{\cD}{\mathcal{D}}

\newcommand{\cR}{\mathcal{R}}

\newcommand{\cG}{\mathcal{G}}

\newcommand{\cI}{\mathcal{I}}

\newcommand{\cT}{\mathcal{T}}

\newcommand{\cX}{\mathcal{X}}
\newcommand{\cY}{\mathcal{Y}}

\newcommand{\onecount}{\textnormal{$1$-count}}
 
\newcommand{\GL}{\mathrm{GL}}

\newenvironment{js}{\noindent\color{red} JS :  }{}

\newcommand{\MFb}{\mathfrak{b}}
\newcommand{\MFc}{\mathfrak{c}}
\newcommand{\MFd}{\mathfrak{d}}

\newcommand{\zeroMat}{0}

\newcommand{\compproblem}[3][]{%
	\par\vspace{0.125cm plus 0.1cm minus 0.05cm}\begin{tabularx}{\textwidth-2\parindent}{lX}%
		\if\relax\detokenize{#1}\relax%
		\else%
		\textnormal{\textbf{Constant:}}&#1\\%
		\fi%
		\textnormal{\textbf{Input:}}&#2\\%
		\textnormal{\textbf{Question:}}&#3\\%
	\end{tabularx}\vspace{0.125cm plus 0.1cm minus 0.05cm}\par%
}

\newcommand{\compalgo}[3][]{%
	\par\vspace{0.125cm plus 0.1cm minus 0.05cm}\begin{tabularx}{\textwidth-2\parindent}{lX}%
		\if\relax\detokenize{#1}\relax%
		\else%
		\textnormal{\textbf{Constant:}}&#1\\%
		\fi%
		\textnormal{\textbf{Input:}}&#2\\%
		\textnormal{\textbf{Output:}}&#3\\%
	\end{tabularx}\vspace{0.125cm plus 0.1cm minus 0.05cm}\par%
}

\newcommand{\NP}{\ensuremath{\mathsf{NP}}\xspace} %

\renewcommand{\P}{\ensuremath{\mathsf{P}}\xspace}

\newcommand{\gen}[1]{\left<\, \mathinner{#1} \,\right>}
\newcommand{\Gen}[2]{\left< \,\mathinner{#1} \mid \mathinner{#2}\,\right>}

\newcommand{\mcomm}[3]{\left[{#1},\kern.1em_{#2}\,\kern.1em {#3} \right]}
\newcommand{\smcomm}[2]{\left[_{#1}\,\kern.1em {#2} \right]}

\newcommand{\QtauPres}{$(Q,\tau)$-presentation}
\newcommand{\tauX}{\tau|_\cX}
\begin{document}
	
	\title[On the complexity of epimorphism testing]
 {On the complexity of epimorphism testing\\  with virtually abelian 
 targets}

\author{Murray Elder}
\address{School of Mathematical and Physical Sciences, University of Technology Sydney, Ultimo NSW 2007, Australia}  
\email{murray.elder@uts.edu.au}

\author{Jerry Shen}
\address{School of Mathematical and Physical Sciences, University of Technology Sydney, Ultimo NSW 2007, Australia}  
\email{qing.shen-1@uts.edu.au}

\author{Armin Wei\ss}
\address{ Institut f\"ur Formale Methoden der Informatik, Universit\"at Stuttgart, 
70569 Stuttgart, Germany}
\email{armin.weiss@fmi.uni-stuttgart.de}

\date{\today}



\begin{abstract}
Friedl and L\"oh (2021, Confl. Math.) prove that testing whether or not there is an epimorphism from a finitely presented group to the direct product of an abelian and a finite group, or to a  virtually cyclic group, is decidable. Here we prove that these problems are $\mathsf{NP}$-complete.

 In addition, we show that testing epimorphism is $\mathsf{NP}$-complete when the target is is the semi-direct product of a finitely generated free-abelian group and a finite group which factors through $\{\pm\text{Id}\}$, 
 thus extending the class of virtually abelian target groups for which decidability of epimorphism is known.

\end{abstract}
\maketitle

\section{Introduction}
Let $\cD, \cT$ be classes of 
groups. The (uniform) \emph{epimorphism problem} from $\cD$ to $\cT$,  denoted $\Epi{\cD}{\cT}$, is the following decision problem.
\compproblem[]{Finite descriptions 
for groups  $G \in \cD$ and $H \in \cT$}{Is there 
a surjective homomorphism
from $G$ to $H$?}

\noindent
Note that an epimorphism in the category of groups is a surjective homomorphism.
We refer to $G\in \cD$ as the \emph{domain group} and $H\in \cT$ is the  \emph{target  group} for the problem.

Remeslennikov
\cite{Remeslennikov} proved for $\mathcal N$ the class of non-abelian nilpotent groups, the  epimorphism problem from $\mathcal N$ to $\mathcal N$
is undecidable, via Hilbert's 10th problem.
Applying the work of Razborov \cite[Theorem 3]{Razborov1985} on equations in free groups, the epimorphism problem from finitely presented  groups to finitely generated free groups is decidable, but without any known complexity bounds (see \cite[Proposition B]{ESW-dihedral}). 
Friedl and L\"oh \cite{FriedlLoh}
considered the epimorphism problem from  finitely presented groups to virtually abelian groups, proving that the problem is decidable when the target is either a virtually cyclic or the direct product of an abelian group and a finite group. 
 Whilst they assert that the algorithms they consider to establish decidability ``will have ridiculous worst-case
complexity'',  in fact we are able to modify their approach slightly to show the following.

\begin{restatable}{theoremx}{ThmMain}\label{ThmMain}
    The epimorphism problem from finitely presented groups
    to the following targets is  in \NP (and therefore \NP-complete):
    \begin{enumerate}
   
        \item virtually cyclic groups

        \item  semi-direct products of a free-abelian group $N$ and a finite group $Q$ where the action of
$Q$ on $N$ factors through $\{\pm\text{Id}\}$ (see \cref{defn:PMonlySemi}) (note that this includes direct products of  abelian  and  finite groups).
      \end{enumerate}
      
\end{restatable}
This is proved in \cref{thm:AbelianDirectProd_is_NP,thm:vC_NP,thm:restrictedPM_NP}.

Kuperberg and Samperton \cite{Kuperberg}  considered the  case of epimorphism from certain 3-manifold groups to finite non-abelian simple groups, in the context of more general questions. 
It follows from their work that the epimorphism problem from a finitely presented group to a fixed finite non-abelian simple group is \NP-hard, and therefore \NP-hardness holds for epimorphism from finitely presented groups to any class containing finite groups.
Thus the main work of this paper is to prove our problems lie in \NP.

Following Friedl and L\"oh, we prove \cref{ThmMain} by reducing the epimorphism problem 
to the following algebraic problems, which we will show are both in \P.  In contrast to Friedl and L\"oh we 
perform this reduction 
by expressing epimorphism in terms of the solvability of equations over groups (see~\Cref{subsec:EquationsNotation}).

For $d \in \N$ let $[1,d]$ denote the set of integers $\{1,\dots,d\}$.
    For $m,n,\ell\in \N$ with $\ell\leq m$ and $M$ an $m\times n$ matrix, let $M|_{\ell}$
    denote the $\ell\times n$ matrix consisting of the bottom $\ell$ rows of $M$. 
We call an $n\times 1$ matrix  an \emph{$n$-vector}, and a matrix (resp. $n$-vector) whose entries are integers an \emph{integer matrix} (resp. \emph{integer $n$-vector}). For an integer matrix $M$ we let $\spanZ{M}$ denote the set of all $\Z$-linear combinations of the columns of $M$ (see~\Cref{subsubsec:MatrixDefn} for additional notation).

\medskip
\MatrixSubspanA
\compproblem{A triple $(A,d,\ell)$ where $A$ is  an $m\times n$ integer matrix, 
$d, \ell \in \N$ with $ \ell\in[0,n-1]$}
{Do there exist  
integer $n$-vectors
 $v_1,\dots,v_d$  such that $Av_i = 0$ for  $i\in[1,d]$ and for the $n\times d$  matrix $V$
whose columns are $v_1,\dots, v_d$, $\spanZ{(V|_\ell)^T} = \Z^d$?}

\medskip
\MatrixSubspanB
\compproblem{A triple $(A,b,\ell)$ where $A$ is an $m\times n$ integer matrix, 
$b$ an integer $m$-vector,
$\ell \in \N$ with $\ell\in[0,n-1]$}
{Does there exist an 
integer $n$-vector
$\nu$ such that $A\nu + b = 0$ and  $\spanZ{(\nu|_\ell)^T}=\Z$?}

We assume integer matrices are given with entries as binary numbers for the purpose of complexity.

\begin{theoremx}\label{thm:MatrixProbsInP}
\MatrixSubspanA and 
    \MatrixSubspanB are in \P.
\end{theoremx}

These two
results are proven as \cref{prop:EQMATSUBSPAN_P,prop:EQMATBSUBSPANZ_P}
below.

\section{Notation and basic facts}\label{sec:Notation}

\subsection{Matrices, basis, span}\label{subsubsec:MatrixDefn}
Let $\Z^{m\times n}$ denote the set of all 
$m\times n$ integer 
matrices (matrices with integer entries),  $\GL(n,\Z)$  the set of invertible $n \times n$ integer matrices, and $\zeroMat_{m, n}$ (or $0$ when the size is clear) the $m\times n$ matrix with all $0$ entries.
If $M\in \Z^{m \times n}$, the matrix 
 $M|_\ell \in \Z^{\ell \times n}$ is the matrix consisting of the bottom $\ell$ rows of $M$, 
 that is, the matrix whose $i$-th row is the $(m-\ell+i)$-th row of $M$. 
 Note that the $\Z$-module $\Z^d$ is identified with $\Z^{d\times 1}$ throughout this paper.
A \emph{$\Z$-linear combination} of $d$-vectors $u_1, \dots, u_n \in \Z^d$ is a $d$-vector of the form $x = c_1 u_1 + \cdots + c_n u_n$ for $c_1, \dots, c_n \in \Z$. The \emph{span} of $u_1, \dots, u_n$ is the set of all $\Z$-linear combinations of $u_1, \dots, u_n \in \Z^d$, which we denote by  $\spanZ{u_1, \dots, u_n}$. If  $M \in \Z^{m\times n}$ we let $\spanZ{M}$ denote the span of the {columns} of $M$. 
For $b\in\Z^m$ we define  $\spanb{M}$ to be the set of all $m$-vectors  of the form $y + b$ for some $y \in \spanZ{M}$.

\subsection{Groups}
We use the notation $1_G$ to denote the identity element of a group $G$,  $[a,b]:= aba^{-1}b^{-1}$ for the {commutator} of two elements $a,b$ of $G$, and $[G,G]$ the commutator subgroup of $G$ (the subgroup consisting of all products of commutators of elements of $G$).  For a group $G$ and two elements $a,b \in G$, ${}^ba = bab^{-1}$
denotes the {conjugation} of $a$ by $b$. If $u,v$ are two different ways to represent the same element of $G$,  we write $u=_G v$.

We denote the infinite cyclic group as $C_\infty$, and the cyclic group of order $n\in \N_+$ as $C_n$. We denote certain classes of groups as follows:
\be
    \item $\Arb$ is the class of finitely presented groups
    \item $\Fin$ is the class of finite groups
    \item $\FreeAb$ is the class of free-abelian groups of finite rank (groups isomorphic to direct products of finitely many copies of $C_\infty$)
    \item  $\vC$ is the class of virtually cyclic groups
    \item $\DPclass$ is  the class of groups of the form $N\times Q$ where $N\in \FreeAb$ and $Q\in \Fin$
    \item $\PMonlyExt$ and $\PMonlySemi$ are restricted classes of extensions to be defined below (\cref{defn:PMonlyExt,defn:PMonlySemi})
\ee

\subsection{Words}
Let $X$ be a set.  We call a finite sequence 
$(x_1,\dots, x_n)$ with $x_i\in X$ a \emph{word} over $X$, and denote it as $x_1\cdots x_n$. 
The set of all words over $X$ is denoted $X^\ast$.
The notation $u(X) = u(x_1,\dots, x_n)$ stands for a word $u$ over the letters $x_1,\dots, x_n$. 
If $Y = \{y_1,\dots,y_p\}, Z=\{z_1,\dots, z_q\}$ are sets, then we may write  $u(Y,z_1,\dots,z_q)=u(y_1,\dots,y_p,z_1,\dots,z_q)=u(Y,Z)$.

For any set $X$ we let 
 $X^{-1}=\{x^{-1}\mid x\in X\}$ be a set of letters with $X\cap X^{-1}=\emptyset$.

If $A$ 
is a set, $H$ is a 
group,
and $\psi\colon A\to H$ is a set map, we define the {induced} monoid homomorphism
from 
$(A\cup A^{-1})^\ast$ to $H$ by \[\psi(a_{1}^{\epsilon_1}\cdots a_{s}^{\epsilon_s})=\psi(a_{1})^{\epsilon_1}\cdots \psi(a_{s})^{\epsilon_s} \]
where $a_i\in A$ and $\epsilon_i=\pm 1$.

\subsection{Presentations}
A group $G\in \Arb$ is given by a finite presentation $\Gen{X}{R}$ where $X$ is a finite set and each $r\in R$  is a word over $X\cup X^{-1}$.  We do not assume $X$ is a subset of $G$, so for example we may have $x,y\in X$ with $x=_Gy$, and we do not assume $G$ has decidable word problem, so, \emph{a priori}, given a presentation $\Gen{X}{R}$,
we have no algorithm to determine whether $x=_Gy$ for $x,y\in X$.
The following well-known lemma is used repeatedly throughout this paper. 

\begin{lemma}[von Dyck's lemma {\cite[Lemma 2.1]{Benc2013}}]\label{lem:vonD}
    If $G$ is presented by $\displaystyle
    \Gen{g_1, \dots, g_n}{r_1, \dots, r_m}$
    where $r_i=r_i(g_1,\dots, g_n)$, and $\psi\colon \{g_1,\dots, g_n\}\to H$ is a set map to a group $H$, then the induced monoid homomorphism $\psi\colon \{g_1^{\pm 1}, \dots, g_n^{\pm 1}\}^\ast\to H$ defines a homomorphism from $G$ to $H$ if and only if $r_i(\psi(g_1),\dots,\psi(g_n))=_{H}1_H$  for $1\leq i\leq m$.
\end{lemma}

\subsection{Complexity}
We assume the reader is familiar with the complexity classes of $\P$ and $\NP$. A problem is \emph{\NP-hard} if every problem in \NP is reducible to it in polynomial time, and a decision problem is \emph{\NP-complete} if it is both in \NP and \NP-hard. For all decision problems and algorithms we assume integer structures (e.g. matrices, $\Z$-modules) are given by a list of integers in binary. If a finite presentation $\Gen{X}{R}$ is given as input we assume the sets $X,R\subseteq (X\cup X^{-1})^*$ are given 
by listing letters and words in unary, so that the size of the input is $|X|+\sum_{r\in R}|r|$.

\subsection{Equations}\label{subsec:EquationsNotation}
Let $\bbX = \{X_1,X_1^{-1},\dots,X_n,X_n^{-1}\}$. An \emph{equation} over a group $G$ is a word
\begin{equation}\label{eq:Equation}
    u(g_1,\dots,g_s,\bbX) 
\end{equation}
where $g_i\in G$ for $i\in[1,s]$ are called  \emph{constants} and $\bbX$ are called
\emph{variables}.
A \emph{system of equations} $(u_i)_{[1,m]}$ is a finite list of  equations $
    u_i(g_1,\dots,g_s,\bbX)$ for $i\in[1,m]$.
A \emph{solution} to a system of equations $(u_i)_{[1,m]}$
is a map $\sigma\colon \bbX \to G$ of the form $\sigma\colon X_i \mapsto h_i,  X_i^{-1} \mapsto h_i^{-1}$ for some $h_i \in G$, $i\in[1,n]$ so that 
\[u_i(g_1,\dots,g_s,\sigma(X_1),\sigma(X_1^{-1}), \dots, \sigma(X_n),\sigma(X_n^{-1}))=_G 1_G \ \ \textnormal{for all} \ \ i\in[1,m].\] 
A system of equations \emph{without constants} is a list of equations of the form     $u_i(\bbX)$ for $i\in[1,m]$.
Note that 
if $G$ is a finitely generated group with finite inverse-closed generating set 
$\cY = \{y_1,\dots,y_s\}$, we may write any equation over $G$ as 
   $ u(\cY,\bbX)$. 

As noted in the introduction, a key step in our reduction from epimorphism to matrix problems 
is the following decision problem for equations over groups.

\medskip
\equationsProbNEW\label{page:EqnSubspan}
\compproblem{a group $N$, variables $\bbX=\{X_1, X_1^{-1},\dots,X_t, X_t^{-1}\}$, $\bbY=\{Y_1, Y_1^{-1},\dots,Y_{\ell}, Y_{\ell}^{-1}\}$, and a finite system of equations over $N$ using variables $\bbX\cup\bbY$.}
{is there a solution $\sigma\colon \bbX\cup\bbY \to N$ such that $\gen{\sigma(Y_1),\dots, \sigma(Y_\ell)} = N$?}

\subsection{Free-abelian groups}
A free-abelian group of rank $d\in\N_+$ is a group isomorphic to the direct product of $d$ copies of the infinite cyclic group $C_\infty$. Such a group
admits the presentation $\Gen{x_1,\dots,x_d}{[x_i,x_j], \forall i,j \in [1,d]}$, which we may write  simply as $\gen{x_1,\dots,x_d}$ when the context is understood. It follows that every element of a free-abelian group $N\cong \gen{x_1,\dots,x_d}$  can be represented uniquely as a word $x_1^{c_1}\cdots x_d^{c_d}$ for some $c_1,\dots, c_d \in \Z$.
It is well known that every free-abelian group is a $\Z$-module which we denote by $\Z^d$, 
with standard basis $\{e_1,\dots,e_d\}$ where 
$e_i=(0,\dots, 0, 1,0,\dots, 0)^T\in \Z^d$ 
with non-zero entry in the $i$-th position, via 
the natural \label{page:natIso}
isomorphism $\varphi\colon N \to \Z^d$  defined by the set map
\begin{equation*}
    \varphi\colon x_i \mapsto e_i,
\end{equation*}
extending to the isomorphism
\begin{equation*}
    \varphi\colon x_1^{c_1}\cdots x_d^{c_d}\mapsto c_1e_1+\cdots + c_de_d.
\end{equation*}

    Let 
 $u=g_0X_1^{\pm 1}g_1\cdots X_n^{\pm 1}g_n$  be an equation over an abelian group $N$ with variables $\bbX$ as in \cref{eq:Equation}.
    Define the \emph{commuted normal form}  of $u$ to be 
    \begin{align*}
        \CNF(u) = X_1^{\alpha_1}\cdots X_n^{\alpha_n}g
    \end{align*} where $\alpha_i=\abs{u}_{X_i}-\abs{u}_{X_i^{-1}}$ for $i\in[1,n]$ and $g=\prod_{i=0}^n g_i$. 
        By commutativity we have that $\sigma\colon \bbX \to N$ is a solution to $(u_i)_{[1,m]}$  if and only $\sigma$ is a solution to $(\CNF(u_i))_{[1,m]}$.

\subsection{Extensions}
A group $H$ is said to be an \emph{$N$ by $Q$ extension} if 
\be\item $N$ is a normal subgroup of $H$
\item there is an isomorphism 
$\psi\colon N\backslash H\to Q$.\ee
It follows that the maps $\iota\colon N \to H$   and $\pi_Q\colon H \to Q$ given by given by $\iota\colon n\mapsto n$ and $\pi_Q \colon g\mapsto \psi(Ng)$ define an exact sequence 
\begin{equation*}
    \{1\} \longrightarrow N \overset{\iota}\longrightarrow  H \overset{\pi_Q}\longrightarrow  Q \longrightarrow  \{1\}.
\end{equation*}
A set $T\subseteq H$ which contains exactly one element of each (right) coset of a subgroup $N$ is called a 
(right) \emph{transversal} for $N$ in $H$ (note that throughout this paper all cosets will be right cosets).
A map $s\colon Q \to H$ is a \emph{transversal map} 
if $s$ is injective and $\pi_Q(s(q)) = q$.
W.l.o.g we assume  $s$ is always chosen so that $s(1_Q)=1_H$ throughout this article.

Given a fixed transversal map $s$, every element $g\in H$ can be written uniquely  as a product $g = ns(q)$ for some $n\in N$ and $q\in Q$, which we call the \emph{normal form} for $g$. The map $\pi_N\colon H \to N$  given by 
$\pi_N(g) = n$ when $g = ns(q)$ is well defined since $s$ is a transversal map. 
Since $N$ is a normal subgroup of $H$, each $s(q)$ acts by conjugation on $N$ as an inner automorphism. Define $\theta_s\colon Q \to \Aut(N)$  by $\theta_s\colon q \mapsto {}^{s(q)}n$. Also since $N$ is normal, for $q_1,q_2 \in Q$ we have  $Ns(q_1)Ns(q_2) = Ns(q_1)s(q_2)$, so  $s(q_1)s(q_2) = ns(q_1q_2)$ for some $n \in N$, and so 
 $s(q_1)s(q_2)s(q_1q_2)^{-1} \in N$.  Define a map $f_s\colon Q \times Q \to N$ by $f_s\colon (q_1,q_2)\mapsto s(q_1)s(q_2)s(q_1q_2)^{-1}$. 
We call the pair  $(\theta_s, f_s)$ the \emph{extension data} for the $N$ by $Q$ extension $H$ with respect to a chosen transversal map $s\colon Q\to H$. Clearly if $Q$ is finite and $N$ is finitely generated then  $(\theta_s, f_s)$ has a finite description.

In the case $s$ is a homomorphism then $s(q_1)s(q_2)s(q_1q_2)^{-1} = 1_N$ and we write $f_s = \constMap$
to denote the trivial map, and $H$ is a semidirect product of $N$ and $Q$ via $\theta_s$. 
For example, if $H$ is isomorphic to the direct product of a group $N$ and a finite group $Q$, $s\colon q\mapsto (1,q)$ is a transversal map which is an injective homomorphism so $f_s=\constMap$ and $\theta_s\colon n\mapsto n$, and we may write every element uniquely in the form $ns(q)=(n,q)$. 

\subsection{Virtually abelian groups}
A standard 
argument shows that if a finitely generated group $H$ has a finite index abelian subgroup, then $H$ contains a normal finite index free-abelian subgroup (by taking the \emph{normal core}, see for example 
\cite[Proposition 2.2]{FriedlLoh}).
Thus we may view every virtually abelian group $H$ as an $N$ by $Q$ extension where $N\in \FreeAb$  and $Q\in \Fin$.

\begin{remark}[Action is determined by $q$ when $N$ is abelian]   \label{rmk:NabelianTransversalIndependent}
If $s_1,s_2$ are two transversal maps from $Q$ to $H$, $Ns_1(q) = Ns_2(q)$ so $s_2(q)^{-1}s_1(q) \in N$.
Then since $N$ is abelian we have 
\begin{align*}
    {}^{s_2(q)}(n)\left({}^{s_1(q)}(n)\right)^{-1} 
    &= s_2(q) n s_2(q)^{-1} s_1(q)n^{-1} s_1(q)^{-1} \\
    &= s_2(q)  n n^{-1} \left(s_2(q)^{-1} s_1(q)\right)  s_1(q)^{-1} = 1_N
\end{align*}
which shows that  the conjugation action does not depend on the choice of transversal.
Thus, when $N$ is abelian we may sometimes write 
${}^{q}n$ rather than ${}^{s(q)}n$, and $\theta$ rather than $\theta_s$.
\end{remark}

Note that the automorphism group  of the infinite cyclic group $\Aut(C_\infty)\cong 
\{ n\mapsto n, n \mapsto n^{-1}\}$, so in a $C_\infty$ by $Q$ extension each $q\in Q$ can act as 
 either ${}^qn=n$  for all $n\in C_\infty$ or ${}^qn=n^{-1}$ for all $n\in C_\infty$.
In this case we say it is a semi-direct product of $C_\infty$ and $Q$ which factors through $\{\pm \text{Id}\}$.

    Now consider the following standard group theory exercise.
Suppose that 
     an $N$ by $Q$ extension $H$ has transversal map  $s\colon Q\to H$ such that  for some  $q \in Q$, ${}^{s(q)}(n) = n^{-1}$ for all $n\in N$.
     Then 
\begin{align*}
    n_1n_2 = s(q)n_1^{-1}s(q)^{-1}s(q)n_2^{-1}s(q)^{-1} =
     s(q)n_1^{-1}n_2^{-1}s(q)^{-1} 
    = (n_1^{-1}n_2^{-1})^{-1} = n_2n_1
\end{align*}
so $N$ is abelian. 
Thus, if we wish to define a class of $N$ by $Q$ extensions where the action of $q\in Q$ is restricted to being either ${}^{s(q)}x=x$ for all $x\in N$ or ${}^{s(q)}x= x^{-1}$ for all $x\in N$ (generalising the class $\vC$), then necessarily $N$ must be abelian.

Motivated by the above exercise,  define two subclasses of virtually abelian groups as follows.

\begin{definition}[$\PMonlyExt$]  \label{defn:PMonlyExt}
    Define $\PMonlyExt$ to be the class of $N$ by $Q$ extensions for a group $H$ which satisfy the following conditions: 
\be\item $N$ is abelian and $Q$ is finite 
\item  there is   a transversal map $s\colon Q\to H$ and a subset $\cI\subseteq Q$ 
so that \be\item 
 ${}^qn=n^{-1}$ for all $n\in N$ when $q\in \cI$
\item   ${}^qn=n$ for all   $n\in N$ when $q\in Q\setminus \cI$. \ee
\ee

\end{definition}
In other words $\theta_s\colon Q \to \Aut(N)$  is completely determined by the subset $\cI$: \[\theta_s(q)=\begin{cases}n\mapsto n^{-1} & q\in \cI\\
n\mapsto n & q\in Q\setminus \cI\\
\end{cases}\] 
so we can specify the extension data for  $H \in \PMonlyExt$ by 
the pair $(\cI,f_s)$, which we call the \emph{special extension data} of $H$.
The class $\PMonlyExt$ includes $\DPclass$ (when $\cI=\emptyset$ and $f_s = \constMap$)
and $\vC$ (when $N\cong C_\infty$).

\begin{definition}[$\PMonlySemi$] \label{defn:PMonlySemi}
    Define $\PMonlySemi$ to be the subclass of $\PMonlyExt$ having special extension data $(\cI, \constMap)$. 
\end{definition}

An example of a group in $\PMonlySemi$ but not $\vC$ or $\DPclass$ is the following semidirect product of $\Z^2$
 and  $C_2$:
 \[H=\Gen{a,b,t}{[a,b]=t^2=1, {}^{t}a=a,{}^{t}b=b},\]  
 Here 
 $\cI=\{t\}$. 

\section{Preliminary results}

\subsection{Finite targets}
We start by observing that $\Epi{\Arb}{\Fin}$ is  in \NP. 
We note that Holt and Plesken \cite[Chapter 7]{HoltP1989perfect}  considered the computational problem of finding epimorphisms onto various classes of finite groups, without explicitly giving a complexity bound, and Friedl and L\"oh \cite[Proposition 5.2]{FriedlLoh}  show  that $\Epi{\Arb}{\Fin}$ is decidable.

We assume that the input to the problem is  a finite presentation $\Gen{g_1,\dots,g_n}{r_1,\dots,r_m}$ for the domain group  $G\in \Arb$ and a  multiplication table for the target group $Q\in \Fin$.

\begin{lemma}
    \label{lem:epi_finite} 
    $\Epi{\Arb}{\Fin}$ is in \NP.
\end{lemma}

See \cite[Lemma~2.1]{ESW-dihedral} for a proof.
Combining \cref{lem:epi_finite} with \cite[Corollary 1.2]{Kuperberg} 
we  have    $\Epi{\Arb}{\Fin}$ is \NP-complete.

\subsection{\QtauPres}\label{subsec:Qtau}
The following is a useful format in which  to present a finitely presented  domain group when considering  epimorphism onto a target which involves a finite group.

\begin{definition}[\QtauPres]\label{defn:Qtau}
    Let $G$ be a finitely presented group, $Q$ a finite group, and  $\tau\colon G\to Q$ an epimorphism from $G$ to $Q$.
    We call $\Gen{\cX\cup \cY}{\cR}$ a \emph{\QtauPres} for $G$ if 
    \be
        \item $\cX,\cY,\cR$ are finite
    
        \item 
        $\cX\subseteq G$ and 
      $\tauX\colon \cX\to Q$ is a bijection
        \item the subgroup $\ker(\tau)$ is generated by $\cY$.
    \ee
\end{definition}

\begin{lemma}\label{lem:buildQtau}
    There is an algorithm which takes input
    \be
        \item a finite presentation $\Gen{g_1, \dots, g_n}{r_1, \dots, r_m}$ for a group $G$
        \item a multiplication table for a finite group $Q$
        \item   a list $(q_{1},\dots, q_{n})\in Q^n$ defining an epimorphism  $\tau\colon G\to Q$ by $\tau(g_j)=q_{j}$, $j\in[1,n]$
    \ee
    and outputs a 
    $(Q,\tau)$-presentation for $G$ which has size polynomial in $(n+m+\abs{Q})$,    which runs in time polynomial in $(n+m+\abs{Q})$.
\end{lemma}

\begin{proof}
    Our procedure is as follows. Initialise $\Lambda = \cX = \cY = \emptyset$, $\cG = \{g_1,\dots,g_n\}$ and $\cR = \{r_1,\dots,r_m\}$. 
    \begin{enumerate}
        \item Set $\cG$ to be $\cG\cup\cG^{-1}$ and $\cR$ to be $\cR\cup\{g_i(g_i^{-1}) : i\in[1,n]\}$.
        \item For each $g \in \cG$: \be\item if $\tau(g) = q \not \in \Lambda$, set $\Lambda = \Lambda\cup\{q\}$, $\cX = \cX \cup \{x_q\}$ and $\cR = \cR \cup \{x_q g^{-1}\}$. \ee
        Since $\tau$ is an epimorphism, after this for-loop we have  $\Lambda$ is a generating set for $Q$, and $\tau(x_q)=q$ for each $x_q\in \cX$.
        \item While $\Lambda \not = Q$:
        \begin{enumerate}
            \item scan the multiplication table for $Q$ to find a triple $(p_1,p_2,q)$ where $p_1,p_2\in \Lambda$, $p_1p_2 = q$, and  $q \in Q \setminus \Lambda$ (such a $q$ must exist as $\Lambda$ is a generating set for $Q$)
            \item set $\Lambda=\Lambda\cup\{q\}$, $\cX = \cX \cup \{x_q\}$ and $\cR = \cR \cup \{x_{p_1}x_{p_2}x_q^{-1}\}$.
        \end{enumerate}
        On termination of this while-loop, we have $\Lambda = Q$ and  $\cX$ is in bijection with $\Lambda = Q$.
        As $\Lambda$ increases in each iteration, the loop is guaranteed to terminate. We have $\Gen{\cX \cup \cG}{\cR}$ presents $G$, and $\tau(x_q)=q$ for each $x_q\in \cX$.
        \item For each $g_i \in \cG$ and pair $(p,q) \in Q \times Q$: \be\item  if $\tau(x_p g_i x_q) =_Q 1$, set $\cY = \cY \cup \{y_{p,i,q}\}$, and $\cR = \cR \cup \{x_pg_ix_qy_{p,i,q}^{-1}\}$. \ee
        After this for-loop have $\Gen{\cX \cup \cY \cup \cG}{\cR}$ presents $G$.
        \item For each $i \in [1,n]$, we have that $\tau(g_i) \in Q$, which means $x_{\tau(g_i)}\in\cX$ and $\tau(x_{\tau(g_i)})=\tau(g_i)$. From this it follows that $\tau(x_{\tau(g_i)^{-1}}g_ix_{1_Q})=1_Q$, which implies $y_{\tau(g_i)^{-1},g_i,1_Q}\in\cY$ and $y_{\tau(g_i)^{-1},g_i,1_Q}=_Gx_{\tau(g_i)^{-1}}g_ix_{1_Q}$. Then we may perform a Tietze transformation to remove $g_i$ from the generating set and replace each occurrence of the letter $g_i$ in each relation by 
        \[x_{\tau(g_i)}y_{\tau(g_i)^{-1},g_i,1_Q}.\]
        After performing this for each $i\in[1,n]$, we
        obtain the presentation  $\Gen{\cX \cup \cY}{\cR}$ for $G$.
    \end{enumerate}
    The time complexity and output length for each subroutine is as follows.
    \begin{enumerate}
        \item  takes $\abs{\cG} = n$ steps, and so from here on setting $\cG$ to be $\cG\cup\cG^{-1}$ gives $2n$ steps when iterating through $\cG$.
        \item  takes $2n$ steps, and adds at most $\abs{Q}$ letters to $\Lambda, \cX$ and $\abs{Q}$ words of length 2 to $\cR$. 
        \item the while-loop  is performed at most $\abs{Q}$ times. Each iteration scans at most $\abs{Q}^2$ entries of the multiplication table, and adds at most $\abs{Q}$ letters to $\Lambda, \cX$ and at most $\abs{Q}$ words of length $3$ to $\cR$.
        \item  takes  $2n\abs{Q}^2$ steps, adds at most $2n\abs{Q}^2$ new letters to $\cY$, and adds at most $2n\abs{Q}^2$ new words of length at most $4$ to $\cR$. 
        \item  takes $2n$ steps, and all steps combined increase the length of relators by at most a factor of $2$ (replacing $g_i$ by a word of length $2$).
    \end{enumerate}
    To show that $\cY$ is a generating set for $\ker(\tau)$, suppose  an element in $\ker(\tau)$ is spelled as 
    \begin{equation*}
        w=g_{i_1}g_{i_2}\cdots  g_{i_k}\in (\cG\cup\cG^{-1})^\ast.
    \end{equation*} 
Let $q_{i_1},\dots, q_{i_{k-1}}\in Q$ so that $q_1=\tau(g_{i_1})^{-1}$ and $q_{i_j}=\tau(x_{q_{j-1}}^{-1}g_{i_j})^{-1}$ for $j \in[2,k-1]$
(recall that $q=\tau(x_q)$ for each $x_q\in \cX$).
Then 
    \begin{align*}
        w &= g_{i_1} \left(x_{q_1} x_{q_1}^{-1}\right) g_{i_2} \left(x_{q_2} x_{q_2}^{-1}\right)  \cdots  \left(x_{q_{k-1}} x_{q_{k-1}}^{-1}\right)  g_{i_k}
        \shortintertext{{so}}
        \tau(w) &= \tau(g_{i_1}x_{q_1}) \tau(x_{q_1}^{-1}  g_{i_2}x_{q_2})\tau(x_{q_2}^{-1}g_{i_3}x_{q_3})  \cdots \tau(
        x_{q_{k-2}^{-1}}
        g_{i_{k-1}}
        x_{q_{k-1}}) 
        \tau(x_{q_{k-1}}^{-1}g_{i_k}) \\
        &= \tau(y_{1,i_1,q_1}) \tau(y_{{q_1}^{-1}, i_2 q_2}) \cdots \tau( y_{q_{k-2}^{-1}, i_{k-1}, q_{k-1}}) \tau(x_{q_{k-1}}^{-1}g_{i_k}) \\
        &= \tau(x_{q_{k-1}}^{-1}g_{i_k}).
    \end{align*}
    Since $\tau(w) = 1_Q$, it follows that $\tau(x_{q_{k-1}}^{-1}g_{i_k}) = 1_Q$, so $x_{q_{k-1}}^{-1}g_{i_k} \in \cY$ (given by $y_{q_{k-1}^{-1},i_k,1} = x_{q_{k-1}}^{-1}g_{i_k}x_{1_Q}$). 
    Thus we have written $w$ as a product of letters from $\cY$, so $\gen{\cY} = \ker(\tau)$.
\end{proof}

\begin{remark}
    Since we do not assume that $\cG\subseteq G$ or that $G$ has decidable word problem, we do not assert that $\cY$ is a subset of $G$ (it may have repetitions), whereas we have ensured that $\cX\subseteq G$ in the proof of the above lemma. 
\end{remark}

\subsection{Epimorphism into extensions}
Recall that if $H$ is an $N$ by $Q$ extension, the map $\pi_Q\colon H\to Q$ is an epimorphism.
 
\begin{lemma}\label{lem:epi_iff_NbyQ_epi}
    Let $G, N \in \Arb$, and $Q \in \Fin$ and $H$ is given by an $N$ by $Q$ extension and a
    transversal map $s$. The following are equivalent:
    
    \be\item there exists an epimorphism from $G$ to $H$ 
   \item there exist homomorphisms $\tau\colon G \to Q$,  $\kappa\colon G \to H$ such that
        \begin{enumerate}
       \item  $\tau$ is surjective
            \item $\kappa(g)= n s(q)$ implies $q=\tau(g)$
             \item for all $ n \in N$ there exists $ g\in \ker(\tau)$ such that $\kappa(g)= n s(1_Q)$.
        \end{enumerate}
   \ee
\end{lemma}

\begin{proof}
  If $\psi\colon G \to H$ is an epimorphism, then $\tau = \pi_Q \circ \psi$ is an epimorphism.
For each $g\in G$ if $\psi(g)=ns(q)$, then $\tau(g) = \pi_Q (n s(q)) = q$ so $\psi=\kappa$ satisfies item~(b), and if $n\in N$ 
then since $\psi$ is surjective there exists $g\in G$ with $\psi(g) = n s(1_Q)$ and $\tau(g)=\pi_Q(ns(1_Q))=1_Q$ so $\psi=\kappa$ satisfies item~(c).

    Conversely, assume there exist $\tau$ and $\kappa$ as described in the lemma. 
    Then for each $ns(q)\in H$ there exists $g_1\in G$ so that $\tau(g_1)=q$, so  $\kappa(g_1)=n_1s(q)$ for some $n_1\in N$, and $g_2\in \ker(\tau)$ so that $\kappa(g_2)=nn_1^{-1}s(1_Q)$. Then $\kappa(g_1g_2)=nn_1^{-1}s(1_Q)n_1s(q)=ns(q)$.
 Therefore, $\kappa$ is a surjective homomorphism from $G$ to  $H$. 
\end{proof}

\begin{remark}\label{rmk:itemc'}
   Item~(c) in the above lemma may be replaced by   \be\item[(c')] 
   for some   $(Q,\tau)$-presentation $\Gen{\cX\cup\cY}{\cR}$ for $G$, 
   for all $ n \in N$ there exists $ w\in (\cY\cup\cY^{-1})^*$ such that $\kappa(w)= ns(1_Q)$\ee since $\ker(\tau)=\gen{\cY}$.
\end{remark}

\begin{example}[Necessity of the conditions in \cref{lem:epi_iff_NbyQ_epi}]
    Let $G = \Gen{x_1,q_1}{[x_1,q_1],q_1^2} \cong \Z \times C_2$, $N=\gen{x_2}\cong \Z$ and $Q=\Gen{q_2}{q_2^2}\cong C_2$.  Then clearly an epimorphism (isomorphism) from $G$ to $N\times Q$ exists.
 Consider the epimorphism $\tau\colon G \to Q$ defined by $\tau(x_1)=q_2, \tau(q_1)=1_Q$.
    For a homomorphism $\kappa\colon G\to N\times Q$ to satisfy condition (2), for all $ n\in N$ there must exist $ g\in \ker(\tau)=\{1_G,q_1\}$
    so that $\kappa(g)=ns(1_Q)=(n,1_Q)$, which is impossible since 
    $N\times \{1_Q\}$ is infinite.
    This example shows 
    that items~(a)--(c) are all required, it is not enough to have an epimorphism $\tau$ and a homomorphism $\kappa$ that do not satisfy (b) and (c).

\end{example}

\section{Direct product targets}
\label{sec:DirectProd}
In this section, we show the epimorphism problem from a finitely presented group to the direct product of a free-abelian group of rank $d$ and a finite group is in \P. We begin by translating the epimorphism problem into the  problem \equationsProbNEW (defined on page~\pageref{page:EqnSubspan}).

\begin{definition}[Presentation to 
 system of equations]\label{defn:PRESEQNA}
    On input a presentation of the form $\Gen{\cX\cup\cY}{\cR}$ where $\cX,\cY,\cR$ are finite, define $\PRESEQNA(\cX,\cY,\cR)$ to be the set of equations constructed as follows. 
    Let 
    $\bbX=\{X_1, X_1^{-1},\dots,X_{\abs{\cX}},X_{\abs{\cX}}^{-1}\}$, $\bbY=\{Y_1, Y_1^{-1},\dots,Y_{\abs{\cY}}, Y_{\abs{\cY}}^{-1}\}$ be  sets of variables and $\zeta\colon \cX\cup\cX^{-1}\cup\cY\cup\cY^{-1} \to \bbX\cup\bbY$ the bijection defined by 
        \[x_j\mapsto X_j, \quad x_j^{-1}\mapsto X_j^{-1}, \quad y_j\mapsto Y_j, \quad y_j^{-1}\mapsto Y_j^{-1}.\]
    Then \PRESEQNA$(\cX,\cY,\cR)$  is the system of equations  $(\zeta(r_i))_{[1,\abs{\cR}]}.$ 
\end{definition}

Note that by definition $\PRESEQNA(\cX,\cY,\cR)$ is a system of equations  without constants.

\begin{lemma}[Epimorphism onto direct products]\label{lem:epi_iff_equations_new}
    Let $G,N \in \Arb$ and $Q \in \Fin$. The following are equivalent:
    \be\item there exists an epimorphism from $G$ to $N \times Q$ 

    \item 
  there exists an epimorphism  $\tau\colon G \to Q$ such that for some
      $(Q,\tau)$-presentation $\Gen{\cX\cup\cY}{\cR}$ for $G$,  \equationsProbNEW returns `Yes' on input $N$ and  
      $\PRESEQNA(\cX,\cY,\cR)$.
    \ee
\end{lemma}

\begin{proof}
    Assume there exists an epimorphism from $G$ to $N\times Q$.
    By \cref{lem:epi_iff_NbyQ_epi} and \cref{rmk:itemc'}, there exist $\tau\colon G \to Q$ an epimorphism and  $\kappa\colon G \to N \times Q$ a homomorphism such that
         \begin{enumerate}
                     \item[(b)] $\kappa(g)= (n, s(q))$ implies $q=\tau(g)$
\item[(c')]for all $ n \in N$ there exists $ w\in (\cY\cup\cY^{-1})^*$ such that $\kappa(w)= (n,1_Q)$.\end{enumerate}
   Define $\sigma\colon \bbX\cup\bbY \to N$ by
   \[ \begin{array}{lll}
        \sigma(X_i) = \pi_N(\kappa(\zeta^{-1}(X_i))) = \pi_N(\kappa(x_i)),  & \quad \quad &  \sigma(X_i^{-1}) = \pi_N(\kappa(x_i))^{-1}, \\
        \sigma(Y_i)= \pi_N(\kappa(\zeta^{-1}(Y_i))) = \pi_N(\kappa(y_i)),&&\sigma(Y_i^{-1})= \pi_N(\kappa(y_i))^{-1}
    \end{array}\] where $x_i,y_i,X_i,Y_i$ are as in \cref{defn:PRESEQNA}.
   Since $\kappa$ a homomorphism, for each $r\in \cR$ we have $ \kappa(r) = (1_N,1_Q) $ which means
    \begin{align*}
       \sigma(\zeta(r)) &= \pi_N(\kappa(r))=1_N
    \end{align*} 
    and we have verified that $\sigma$ is a solution to $\PRESEQNA(\cX,\cY,\cR)$.

For each  $ n\in N$ there exists $ w\in (\cY\cup\cY^{-1})^*$ such that $\kappa(w)= (n,1_Q)$, and 
 for each $n\in N$ there exists $\zeta(w) \in \bbY^\ast$  so that $\sigma(\zeta(w))= \pi_N(\kappa(w)) = n$, which means 
 $\gen{\sigma(Y_1), \dots, \sigma(Y_\ell)} = N$, and \equationsProbNEW returns `Yes'.

    Conversely, assume there exists an epimorphism  $\tau\colon G \to Q$ such that 
     for some $(Q,\tau)$-presentation $\Gen{\cX\cup\cY}{\cR}$ for $G$,  \equationsProbNEW returns `Yes' on input $N$ and  system of equations $\PRESEQNA(\cX,\cY,\cR)$, which means 
there is a solution $\sigma\colon \bbX\cup\bbY \to N$ such that 
$\sigma(\zeta(r))=1_N$ for each $r\in\cR$  and $\gen{\sigma(Y_1),\dots,\sigma(Y_{\abs{\cY}})} 
= N$. 
    
    Define a set map $\kappa\colon \cX\cup\cY \to N \times Q$ by 
    \begin{equation*}
        \kappa\colon \begin{cases}
            x \mapsto \left(\sigma(X),\tau(x)\right) & x \in \cX \\
            y \mapsto \left(\sigma(Y),\tau(y)\right) 
            & y \in \cY. 
        \end{cases}
    \end{equation*} 
with $\kappa\colon (\cX\cup\cY\cup\cX^{-1}\cup \cY^{-1})^*\to N\times Q$ the induced monoid homomorphism. 
 Then for each $r\in\cR$ where $r = v_1\cdots v_k$ with $v_i\in \cX\cup\cY\cup\cX^{-1}\cup\cY^{-1}$ we have 
    \begin{align*}
  \kappa(r)     
        = \kappa(v_1)\cdots \kappa(v_k) 
        &= (\sigma(\zeta(v_1)),\tau(v_1))\cdots (\sigma(\zeta(v_k)),\tau(v_k)) \\
        &= (\sigma(\zeta(v_1))\cdots (\sigma(\zeta(v_k)), \tau(v_1)\cdots \tau(v_k)) \\
        &= (\sigma(\zeta(v_1\cdots v_k)),\tau(v_1\cdots v_k)) 
        = (\sigma(\zeta(r)),\tau(r)) 
        = (1_N,1_Q) 
    \end{align*} where $\tau(r)=1_Q$ since $\tau$ is a homomorphism, so 
    by \cref{lem:vonD} $\kappa\colon G\to N\times Q$ is a homomorphism.

For any $g\in G$ there exists $w\in (\cX\cup\cY\cup\cX^{-1}\cup\cY^{-1})^*$ with $g=_Gw$. Then $\kappa(g)=\kappa(w)=(\sigma(w),\tau(w))=(\sigma(w), \tau(g))$ so $\kappa(g)=(n,q)$ implies $q=\tau(g)$.

Since $\gen{\sigma(Y_1),\dots,\sigma(Y_{\abs{\cY}})} = N$,  for each $n \in N$ there exists $w \in \bbY^\ast$ such that $\sigma(w) = n$, so for each $n \in N$ there exists $\zeta^{-1}(w)\in\gen{\cY}$ so that   $\sigma(\zeta(\zeta^{-1}(w))) = \sigma(w) = n$.

Having found homomorphisms $\tau,\kappa$ and established conditions (a), (b) and (c') as in \cref{lem:epi_iff_NbyQ_epi} and \cref{rmk:itemc'}, we have shown the existence of an epimorphism from $G$ to $N \times Q$.
\end{proof}

For the remainder of this section we assume $N$ is free-abelian of finite rank.

\begin{definition}[System of equations to matrix system]\label{defn:EQN_TO_MAT}
    Let
   \be
        \item $d, t,\ell, m\in\Z$
        \item $N = \gen{x_1,\dots,x_d} \in \FreeAb$
        \item $\bbX=\{X_1, X_1^{-1},\dots,X_t, X_t^{-1}\}$, $\bbY=\{Y_1,Y_1^{-1},\dots,Y_{\ell},Y_{\ell}^{-1}\}$ 
        \item $v_i\in(\bbX\cup\bbY)^\ast$ for $i \in [1,m]$
        \item $\MFc_i \in N$ with $\MFc_i = x_1^{b_{(i,1)}}\cdots x_d^{b_{(i,d)}}$ for $i \in [1,m]$   
        \item $(u_i)_{[1,m]}$ be a system of equations in $N$ where each equation is of the form $u_i = v_i\MFc_i$.
   \ee
For each $i\in[1,m]$, the commuted normal form 
(defined in Subsection~\ref{sec:Notation}) of $u_i$ is \[\CNF(u_i) = X_1^{\alpha_{(i,1)}}\dots X_t^{\alpha_{(i,t)}} Y_1^{\beta_{(i,1)}} \dots Y_\ell^{\beta_{(i,\ell)}}\MFc_i\]  where 
        \[\alpha_{(i,j)}=\abs{v_i}_{X_j}-\abs{v_i}_{X_j^{-1}}\ \ \textnormal{and}\ \  \beta_{(i,k)}=\abs{v_i}_{Y_k}-\abs{v_i}_{Y_k^{-1}}\] for $j\in[1,t], k\in[1,\ell]$.
Define \[\EQNMAT(d,t,\ell, (u_i)_{[1,m]},(\MFc_i)_{[1,m]}))\] to be the triple $(A,B,\ell)$ with $A \in \Z^{m \times (t + \ell)}$, $B \in \Z^{m\times d}$,  
    \begin{align*}
        A = \begin{pmatrix}
            \alpha_{(1,1)} & \cdots & \alpha_{(1,t)} & \beta_{(1,1)} & \cdots & \beta_{(1,\ell)} \\
            \vdots & \ddots & \vdots & \vdots & \ddots & \vdots \\
            \alpha_{(m,1)} & \cdots & \alpha_{(m,t)} & \beta_{(m,1)} & \cdots & \beta_{(m,\ell)} 
        \end{pmatrix}, B = \begin{pmatrix}
            b_{(1,1)} & \cdots & b_{(1,d)}  \\
            \vdots & \ddots & \vdots  \\
            b_{(m,1)} & \cdots & b_{(m,d)} 
        \end{pmatrix}.
    \end{align*} 
\end{definition}

\begin{definition}[Matrices to system of equations]
    Given $\ell\in \Z$, matrices $A \in \Z^{m \times n}$ and $B \in \Z^{m \times d}$  where
    \begin{align*}
        A = \begin{pmatrix}
            a_{(1,1)} & \cdots & a_{(1,n)} \\
            \vdots & \ddots & \vdots \\
            a_{(m,1)} & \cdots & a_{(m,n)} 
        \end{pmatrix}, B = \begin{pmatrix}
            b_{(1,1)} & \cdots & b_{(1,d)}  \\
            \vdots & \ddots & \vdots  \\
            b_{(m,1)} & \cdots & b_{(m,d)} 
        \end{pmatrix},
    \end{align*}
    define $\MATEQN(A,B,\ell)$ to be the quintuple $(d,n-\ell,\ell,(u_i)_{[1,m]}, (\MFc_i)_{[1,m]})$ where $(u_i)_{[1,m]}$ is a system of equations with variables
    \[\{X_1,X_1^{-1},\dots,X_{n-\ell},X_{n-\ell}^{-1},Y_1,Y_1^{-1},\dots,Y_\ell,Y_\ell^{-1}\}\] 
    over a free-abelian group
    $N=\gen{x_1,\dots, x_d}$ of rank $d$ where each $u_i$ is of the form 
    $u_i=v_i\MFc_i$,   \[v_i=   
         X_1^{a_{(i,1)}}\cdots X_{n-\ell}^{a_{(i,n-\ell)}}Y_1^{a_{(i,n-\ell+1)}}\cdots Y_\ell^{a_{(i,n)}}\]
    and 
    $\MFc_i = x_1^{b_{(i,1)}}\dots x_d^{b_{(i,d)}} \in N$.
\end{definition}

The following is immediate from the definitions.
\begin{lemma}\label{lem:BuildInP}
    The following computations can be achieved in polynomial time.
    \be
    \item 
    On input finite sets $\cX,\cY,\cR\subseteq  (\cX\cup\cY\cup\cX^{-1}\cup\cY^{-1})^\ast$, compute $\PRESEQNA(\cX,\cY,\cR)$
    \item On input $d,t,\ell\in\Z$, 
    a system of equations $(u_i)_{[1,m]}$ with each $u_i=v_i\MFc_i$ where $v_i$ is a word in variables $\{X_1,X_1^{-1},X_t,X_t^{-1}, Y_1,Y_1^{-1},\dots,Y_\ell, Y_\ell^{-1}\}$ and $\MFc_i = x_1^{b_{(i,1)}}\cdots x_d^{b_{(i,d)}}$, compute $\EQNMAT(d, t,\ell,(u_i)_{[1,m]}, (\MFc_i)_{[1,m]})$
    \item On input $A \in \Z^{m\times n}, B \in \Z^{n \times d}$ and $\ell \in \Z$, compute $\MATEQN(A,B,\ell)$.
\ee
\end{lemma}

\begin{lemma}\label{NewLemma:EQN_IFF_EQMATSUBSPAN}
    Let \be\item $N \in \FreeAb$ have rank $d$
    \item  $\bbX=\{X_1, X_1^{-1},\dots,X_t, X_t^{-1}\}$,  $\bbY=\{Y_1,Y_1^{-1},\dots,Y_{\ell},Y_{\ell}^{-1}\}$
    \item  $(u_i)_{[1,m]}$  be a system of equations in $N$ without constants, with each  $u_i\in(\bbX\cup\bbY)^\ast$
    \item  $(A, 0, \ell) = \EQNMAT(d,t,\ell,(u_i)_{[1,m]},(1_N)_{[1,m]}).$\ee 
    The following are equivalent:
    \be\item \equationsProbNEW returns `Yes' on input $N$ and $(u_i)_{[1,m]}$ 
    \item  \MatrixSubspanA returns `Yes' on input $(A,d,\ell)$.
    \ee
\end{lemma}

\begin{proof}
    Let 
    \begin{equation*}
         A = \begin{pmatrix}
            \alpha_{(1,1)} & \cdots & \alpha_{(1,t)} & \beta_{(1,1)} & \cdots & \beta_{(1,\ell)} \\
            \vdots & \ddots & \vdots & \vdots & \ddots & \vdots \\
            \alpha_{(m,1)} & \cdots & \alpha_{(m,t)} & \beta_{(m,1)} & \cdots & \beta_{(m,\ell)} 
        \end{pmatrix}\in \Z^{m\times (t+\ell)}
    \end{equation*}  so that 
   for each $i\in [1,m]$ \[\CNF(u_i) = X_1^{\alpha_{(i,1)}}\dots X_t^{\alpha_{(i,t)}} Y_1^{\beta_{(i,1)}} \dots Y_\ell^{\beta_{(i,\ell)}}\]  where 
        \[\alpha_{(i,j)}=\abs{u_i}_{X_j}-\abs{u_i}_{X_j^{-1}}\ \ \textnormal{and}\ \  \beta_{(i,k)}=\abs{u_i}_{Y_k}-\abs{u_i}_{Y_k^{-1}}\] for $j\in[1,t], k\in[1,\ell]$.

By definition,    \equationsProbNEW on input $N$  and $(u_i)_{[1,m]}$ returns `Yes' if and only if 
there is a solution $\sigma\colon \bbX\cup\bbY \to N$ to  $(u_i)_{[1,m]}$ (and to  $(\CNF(u_i))_{[1,m]}$)  and $\gen{\sigma(Y_1), \dots,\sigma(Y_\ell)} = N$.

We  may write the solution as 
    \begin{equation*}
        \sigma\colon \begin{cases}
            X_j \mapsto x_1^{c_{(j,1)}}\cdots x_d^{c_{(j,d)}},\quad  X_j^{-1} \mapsto (x_1^{c_{(j,1)}}\cdots x_d^{c_{(j,d)}})^{-1}  & j \in [1,t] \\
            Y_j \mapsto x_1^{c_{(t+j,1)}} \cdots x_d^{c_{(t+j,d)}}, \quad Y_j^{-1} \mapsto (x_1^{c_{(t+j,1)}} \cdots x_d^{c_{(t+j,d)}})^{-1} &j \in [1,\ell]
        \end{cases}
    \end{equation*} 
 for some values $c_{j,k}\in\Z$. Let \[V=\begin{pmatrix}
        c_{(1,1)} & \cdots & c_{(1,d)}\\ 
        \vdots & \ddots & \vdots \\ 
        c_{(t+\ell,1)} & \cdots & c_{(t+\ell,d)}
    \end{pmatrix} \in \Z^{(t+\ell)\times d},\] for $i\in[1,d]$ let $v_i \in \Z^{t+\ell}$ denote the $i$-th column of $V$, and for $i\in[1,\ell]$ let $\mu_i$  denote  the $i$-th column of $(V_\ell)^T$.
  By direct calculation  for each $i\in[1,m]$ we have
    \begin{align*}
        \sigma(u_i) &= x_1^{\sum_{j = 1}^t c_{(j,1)}
        \alpha_{(i,j)}        
        + \sum_{j = 1}^\ell c_{(t+j,1)}
        \beta_{(i,j)}}        
        \cdots x_d^{\sum_{j = 1}^t c_{(j,d)}
        \alpha_{(i,j)}        
        + \sum_{j = 1}^\ell c_{(t+j,d)}
        \beta_{(i,j)}}.
    \end{align*} Then 
 for $i\in[1,m]$,     $\sigma(u_i)=1_N$  
     if and only if 
    \begin{equation}\label{eqn:DPeqnSum_0}
        \sum_{j = 1}^t c_{(j,k)} \alpha_{(i,j)} + \sum_{j = 1}^\ell c_{(t+j,k)} \beta_{(i,j)}= 0
    \end{equation} for each $k \in [1,d]$.

  Then   $\sigma\colon \bbX\cup\bbY \to N$ is a solution to  $(\CNF(u_i))_{[1,m]})$ and $\gen{\sigma(Y_1), \dots,\sigma(Y_\ell)} = N$ if and only if 
    \begin{align*}
        AV &=  \begin{pmatrix}
            \alpha_{(1,1)} & \cdots & \alpha_{(1,t)} & \beta_{(1,1)} & \cdots & \beta_{(1,\ell)} \\
            \vdots & \ddots & \vdots & \vdots & \ddots & \vdots \\
            \alpha_{(m,1)} & \cdots & \alpha_{(m,t)} & \beta_{(m,1)} & \cdots & \beta_{(m,\ell)} 
        \end{pmatrix} \begin{pmatrix}
        c_{(1,1)} & \cdots & c_{(1,d)}\\ 
        \vdots & \ddots & \vdots \\ 
        c_{(t+\ell,1)} & \cdots & c_{(t+\ell,d)}
    \end{pmatrix} \\
    &= \begin{pmatrix}
        \sum_{j = 1}^t c_{(j,1)} \alpha_{(1,j)} + \sum_{j = 1}^\ell c_{(t+j,1)} \beta_{(1,j)} & \cdots & \sum_{j = 1}^t c_{(j,d)} \alpha_{(1,j)} + \sum_{j = 1}^\ell c_{(t+j,d)} \beta_{(1,j)} \\
        \vdots & \ddots & \vdots \\
        \sum_{j = 1}^t c_{(j,1)} \alpha_{(m,j)} + \sum_{j = 1}^\ell c_{(t+j,1)} \beta_{(m,j)} & \cdots & \sum_{j = 1}^t c_{(j,d)} \alpha_{(m,j)} + \sum_{j = 1}^\ell c_{(t+j,d)} \beta_{(m,j)}
        \end{pmatrix} \\
    &= \begin{pmatrix}
        0 & \cdots & 0 \\
        \vdots & \ddots & \vdots \\
        0 & \cdots & 0
        \end{pmatrix} \textnormal{ by \cref{eqn:DPeqnSum_0}},
    \end{align*}
  or equivalently $\sigma\colon \bbX\cup\bbY \to N$ is a solution to  $(\CNF(u_i))_{[1,m]})$ and $\gen{\sigma(Y_1), \dots,\sigma(Y_\ell)} = N$ if and only if  $Av_i = 0$ for $i \in [1,d]$
   and $\gen{\sigma(Y_1), \dots,\sigma(Y_\ell)} = N$.

Recall the natural isomorphism $\varphi\colon N \to \Z^d$.
For each $i\in[1,\ell]$ we have
    \begin{align*}
        \varphi(\sigma(Y_i)) &= \varphi(x_1^{c_{(t+i,1)}}\cdots x_d^{c_{(t+i,d)}})  \\
        &= c_{(t+i,1)} e_1 + \cdots + c_{(t+i,d)} e_d  = \mu_i \label{eqn:varphiYj=muj}.
    \end{align*}  

    Then $Av_i = 0$ for $i \in [1,d]$
   and $\gen{\sigma(Y_1), \dots,\sigma(Y_\ell)} = N$ if and only if $Av_i = 0$ for $i \in [1,d]$ and
    for each $h\in N$ 
  there exists $w \in \bbY^\ast$ such that $\sigma(w) = h$. This holds if and only if $Av_i = 0$ for $i \in [1,d]$ and
    for each $z \in \Z^d$  
    there exists $w \in \bbY^\ast$ such that $\varphi(\sigma(w)) =z$.
    Write $\CNF(w)=Y_1^{b_1}\cdots Y_\ell^{b_\ell}$.
    Then 
    \begin{align*}
        z = \varphi(\sigma(w))
        &= \varphi(\sigma(Y_1)^{b_1}\cdots \sigma(Y_\ell)^{b_\ell}) \\
        &= b_1\mu_1 + \cdots + b_\ell \mu_\ell 
    \end{align*}
    Therefore, $Av_i = 0$ for $i \in [1,d]$
    and $\gen{\sigma(Y_1), \dots,\sigma(Y_\ell)} = N$ if and only if 
    $Av_i = 0$ for $i \in [1,d]$ and 
    \[\spanZ{(V|_\ell)^T} =\spanZ{\mu_1,\dots,\mu_\ell} = \Z^d\] which is true if and only if \MatrixSubspanA returns `Yes'.   
\end{proof}

Combining the above results with the fact that will be proved in \cref{sec:LA} that \MatrixSubspanA can be decided in polynomial time gives the following.

\begin{theorem}\label{thm:AbelianDirectProd_is_NP}
Let $\cT$ be a subclass of $\DPclass$. Then 
   $\Epi{\Arb}{\cT}$ is in \NP.
\end{theorem}

\begin{proof}
    Let $G \in \Arb$, $N \in \FreeAb$ and $Q \in \Fin$. Using \cref{lem:epi_iff_equations_new} we may verify the existence of an epimorphism from $G$ to $N \times Q$  by verifying that:
    \begin{enumerate}
        \item[(i)] there exists an epimorphism $\tau\colon G \to Q$ 
        \item[(ii)] for some $(Q,\tau)$-presentation $\Gen{\cX\cup\cY}{\cR}$ for $G$, the output to \equationsProbNEW is `Yes' on input $N$ and $\PRESEQNA(\cX,\cY,\cR)$.
    \end{enumerate}

    On input $G = \Gen{\cG}{\cR}$, $d \in \N_+$ encoding a  free-abelian group $N$ of rank $d$,  and a multiplication table encoding a finite group $Q$, the following procedure solves our problem:
    \begin{enumerate}
        \item guess a set map $\tau\colon \cG \to Q$ and verify that its extends to an epimorphism $\tau\colon G \to Q$
        \item construct a $(Q,\tau)$-presentation $\Gen{\cX\cup\cY}{\cR}$
        \item construct a system of equations from $\PRESEQNA(\cX,\cY,\cR)$ denoted as $(u_i)_{[1,m]}$.
        \item construct the triple $(A,\zeroMat_{m,d},\abs{\cY}) = \EQNMAT(d,\abs{\cX},\abs{\cY},(u_i)_{[1,m]},(1_N)_{[1,m]})$ where $A \in \Z^{m\times (\abs{\cX}+\abs{\cY})}$ 
        \item return the output of input 
        \MatrixSubspanA on input $(A,d,\abs{\cY})$.

    \end{enumerate}
Step~(1) verifies the existence of Condition~(i). step~(1) to~(3) builds the necessary data to solve Condition~(ii). \cref{NewLemma:EQN_IFF_EQMATSUBSPAN} states that to solve \equationsProbNEW,  we can solve \MatrixSubspanA on input 
$(A, 0, \abs{\cY})$ constructed in step~(4). Thus we solve \MatrixSubspanA in step~(5) and output the solution.
    
    The time complexity of the procedure is as follows:
    \begin{enumerate}
        \item We verify the correct $\tau$ in \NP by \cref{lem:epi_finite}; this is the only non-deterministic step of our algorithm.
        \item A construction of $(Q,\tau)$-presentation in \P exists by \cref{lem:buildQtau}.
        \item $\PRESEQNA(\cX,\cY,\cR)$ is a polynomial time construction by \cref{lem:BuildInP}.
        \item $\EQNMAT(d,\abs{\cX},\abs{\cY},(u_i)_{[1,m]},(1_N)_{[1,m]})$ is a polynomial time construction by \cref{lem:BuildInP}.
        \item \MatrixSubspanA is solved \P by \cref{prop:EQMATSUBSPAN_P} (Subsection~\ref{subsec:EQMATSUBSPAN}).
    \end{enumerate}
    Thus, our algorithm is in \NP. 
\end{proof}

\section{Virtually cyclic targets} 
\label{sec:vcyc}

In this section, we show that the epimorphism problem from a finitely presented group to a virtually cyclic group is in \P. 
We again begin by translating epimorphism  
to an equations problem. 

Recall that $\PMonlyExt$ is the class of  $N$ by $Q$ extensions such that $Q$ is finite, $N$ is abelian, and there exists a transversal map $s$ and a subset $\cI \subseteq Q$ so that
\[\theta_s(q)=\begin{cases}n\mapsto n^{-1} & q\in \cI\\
n\mapsto n & q\in Q\setminus \cI\\
\end{cases}\] so for all $n\in N$,  
\[{}^{s(q)}n=s(q)ns(q)^{-1}=\begin{cases}n^{-1} & q\in \cI\\
 n & q\in Q\setminus \cI.
\end{cases}\] 
We assume the data for a group in $\PMonlyExt$ is given as a group $N \in \FreeAb$, $Q \in \Fin$ and special extension data $(\cI,f_s)$.

The next three definitions 
introduce some notation that will be useful in our proofs below. 

\begin{definition}\label{defn:TildeF}
Let $H\in\PMonlyExt$ with $N,Q,(\cI,f_s)$ as above.
    For $k\geq 2$ define $\tildef_k\colon Q^k\to N$ by     \begin{equation*}
        \tildef_k(a_1,\dots, a_k) = f_s(a_1,a_2)f_s(a_1a_2,a_3)\cdots f_s(a_1\cdots a_{k-1}, a_k).
    \end{equation*}
\end{definition}

\begin{definition}[Left $A$-count] 
    Let $A,B$ be two disjoint sets and 
    $w = v_1\cdots v_n$ with $v_i \in A\cup B\cup A^{-1}\cup B^{-1}$. For each $p\in[1,n]$ 
    let 
    \[ k_p = \abs{v_{1} \cdots v_{p-1}}_{A} -  \abs{v_{1} \cdots v_{p-1}}_{A^{-1}} \] which we call the \emph{left $A$-count of $w$ at position $p$}.     Define  $\sgn(w,A,p) = (-1)^{k_p}$. 
\end{definition}
The digit $\sgn(w,A,p)$ encodes whether the number of letters from $A$ minus the number of letters from $A^{-1}$ (ignoring all letters from $B$) in the length $p-1$ prefix of $w$ is odd or even.

\begin{definition}\label{defn:GammaTechnical}
Let $\Gen{\cX\cup\cY}{\cR}$ be a \QtauPres\ for a group $G$,  $\cI\subseteq \cX$, and let 
$\cI_\cX=\{x\in 
\cX\mid\tau(x)\in \cI\}$ denote the preimage of $\cI$ under the bijection  $\tauX\colon \cX\to Q$. For 
 $r\in (\cX\cup \cY\cup\cX^{-1}\cup \cY^{-1})^*$, 
 define  $\gamma(r)$ to be the word obtained by raising the $p$-th letter of $r$ to the power of $\sgn(r, \cI_\cX, p)$
for each $p\in[1,\abs{r}]$.
\end{definition}

\begin{example}
    If $r=x_1y_1x_3^{-2}y_2x_1^{-1}x_2^{-1}y_1x_1$ and $\cI_\cX=\{x_2,x_3\}$ then \[\gamma(r)= x_1y_1x_3^{-1}x_3y_2x_1^{-1}x_2^{-1} y_1^{-1}x_1^{-1}\] 
\end{example}

The purpose of defining $\gamma$ in this way will become evident 
in the proof of \cref{lem:epi_iff_equations_pmonly}.

Next we define a way to construct a system of equations from a presentation which will be useful for analysing epimorphism onto the class $\PMonlyExt$, 
analogously to the construction in \cref{defn:PRESEQNA} for direct products.

\begin{definition}[Presentation to system of equations for $\PMonlyExt$]\label{defn:PRESEQNB}
Let $H\in \PMonlyExt$ be an $N$ by $Q$ extension with special extension data $(\cI,f_s)$,
$\Gen{\cX\cup\cY}{\cR}$ a \QtauPres\ for a group $G$, $\cI\subseteq \cX$, 
$\bbX,\bbY$  alphabets with $\cX\cup\cX^{-1}$ in bijection with $\bbX$ and $\cY\cup\cY^{-1}$ in bijection with $\bbY$ via 
\[\zeta\colon x_j\mapsto X_j, \quad  x_j^{-1}\mapsto X_j^{-1}, \quad y_j\mapsto Y_j, \quad y_j^{-1}\mapsto Y_j^{-1},\]
 $\cI_\cX=\{x\in \cX\mid \tau(x)\in \cI\}$ the preimage of $\cI$ under the bijection  $\tauX\colon \cX\to Q$,  $\gamma$ as in \cref{defn:GammaTechnical}, and $\tildef_k$ as in \cref{defn:TildeF}.
    For $i\in[1,\abs{\cR}]$ assume each $r_i\in \cR$ has the form  \[r_i=v_{i,1}\dots v_{i,\abs{r_i}}\] with $v_{i,j}\in \cX\cup\cY\cup\cX^{-1}\cup\cY^{-1}$.
    Define $\PRESEQNB(\tau, \cX,\cY,\cR,\cI, f_s)$ to be the system of equations $(u_i)_{[1,\abs{\cR}]}$ where  \[u_i =\zeta(\gamma(r_i)) \tildef_{\abs{r_i}}(\tau(v_{i,1}),\dots,\tau(v_{i,\abs{r_i}})).\] 
\end{definition}

Note that by definition $\PRESEQNB(\tau, \cX,\cY,\cR,\cI, f_s)$ is a system of equations with variables in $\bbX\cup\bbY$ with a single constant.  It is clear from the above definitions that it can be written in polynomial time. We will now show how it arises in the context of epimorphism to $\PMonlyExt$.

Recall the decision problem 
 \equationsProbNEW.

\begin{lemma}\label{lem:epi_iff_equations_pmonly}
    Let $G\in \Arb$, $N \in \FreeAb$, $Q \in \Fin$ and $H 
    \in \PMonlyExt$ where $H$ is a $N$ by $Q$ extension with special extension data $(\cI,f_s)$. The following are equivalent. 
    \begin{enumerate}\item there exists an epimorphism from $ G$ to $H$ 
    \item  there exists an epimorphism $\tau\colon G \to Q$ 
     such that for some \QtauPres\ $\Gen{\cX\cup\cY}{\cR}$ for $G$, \equationsProbNEW returns `Yes' on input $N$ and   
     $\PRESEQNB(\tau, \cX,\cY,\cR,\cI, f_s).$
    \end{enumerate}
\end{lemma}

\begin{proof} 
   Assume there exists an epimorphism from $G$ to $H$.
    By \cref{lem:epi_iff_NbyQ_epi} and \cref{rmk:itemc'}, there exist $\tau\colon G \to Q$ an epimorphism and  $\kappa\colon G \to H$ a homomorphism such that
         \begin{enumerate}
                     \item[(b)] $\kappa(g)= n s(q)$ implies $q=\tau(g)$
\item[(c')]for all $ n \in N$ there exists $ w\in (\cY\cup\cY^{-1})^*$ such that $\kappa(w)= ns(1_Q)$\end{enumerate}
where $\Gen{\cX\cup\cY}{\cR}$ is some \QtauPres\ for $G$.

For $r \in \cR$ let $r = v_1 \cdots v_k$ with $v_i \in (\cX\cup\cY)\cup(\cX\cup\cY)^{-1}$.
Note that  for each $v_i$ we have 
\begin{align*}
\kappa(v_i) &= \pi_N(\kappa(v_i))s(\tau(v_i))= \pi_N(\kappa(v_i))s(v_i')  \end{align*}
where $v_i'=\tau(v_i)$.
Since $\kappa$ is a homomorphism and $r$ is a relation we have 
\begin{align*}
       1_N= \kappa(r) &= \kappa(v_1) \kappa(v_2) \cdots \kappa(v_k) 
        = \pi_N(\kappa(v_1))s(v_1') \pi_N(\kappa(v_2))s(v_2') \cdots \pi_N(\kappa(v_k))s(v_k').  
    \end{align*}
By inserting $s(v_1')^{-1}s(v_1')$ after $\pi_N(\kappa(v_2))$, $s(v_2')^{-1}s(v_1')^{-1}s(v_1')s(v_2')$ after $\pi_N(\kappa(v_3))$ and so on, we obtain 
\begin{align}\label{eqn:ExpansionM1} 
       1_N
        &= \pi_N(\kappa(v_1))[{}^{s(v_1')} \pi_N(\kappa(v_2))][{}^{s(v_1')s(v_2')}\pi_N(\kappa(v_3))] \cdots [{}^{s(v_1')\cdots s(v_{k-1}')}\pi_N(\kappa(v_k))] \ s(v_1')\cdots s(v_k')
    \end{align}

Let us first deal with the term $s(v_1')\cdots s(v_k')$ at the end of  \cref{eqn:ExpansionM1}.
By definition of the map $f_s\colon Q\times Q\to N$ 
we have 
\begin{align*}s(v_1')s(v_2')&=f_s(v_1',v_2')s(v_1'v_2')\end{align*}
then 
\begin{align*}s(v_1')s(v_2')s(v_3')=f_s(v_1',v_2')s(v_1'v_2')s(v_3')
&=f_s(v_1',v_2')f_s(v_1'v_2',v_3')s(v_1'v_2'v_3')\\
&=\tildef_3(v_1',v_2',v_3')s(v_1'v_2'v_3').
\end{align*}
Repeating this we obtain
\begin{align}\label{eqn:sProductFtilde}
s(v_1')\cdots s(v_k')
&=\tildef_k(v_1',\dots, v_k')s(v_1'\cdots v_k') \nonumber\\
&=\tildef_k(v_1',\dots, v_k')s(\tau(v_1\cdots v_k)) \nonumber\\
&=\tildef_k(v_1',\dots, v_k')s(\tau(r)) \nonumber\\
&=\tildef_k(v_1',\dots, v_k')
\end{align}
since $r\in \cR$ and $\tau$ is a homomorphism so $s(\tau(r))=s(1_Q)=1_H$.

    Now we will deal with the term 
    \[
\pi_N(\kappa(v_1))[{}^{s(v_1')} \pi_N(\kappa(v_2))][{}^{s(v_1')s(v_2')}\pi_N(\kappa(v_3))] \cdots [{}^{s(v_1')\cdots s(v_{k-1}')}\pi_N(\kappa(v_k))]
\]
at the start of \cref{eqn:ExpansionM1}.
Recall from \cref{defn:GammaTechnical} that $\cI_\cX$ is the preimage of $\cI\subseteq Q$  under the bijection $\tauX$.
If $v_i\in \cY\cup\cY^{-1}$ then $v_i'=\tau(v_i)=1_Q$ so conjugation  by $s(v_i')$ sends $n\mapsto n$, and  
conjugation by $s(\tau(v_i))$  sends $n\mapsto n$ if $v_i\in \cX\setminus \cI_X$, 
and $n\mapsto n^{-1}$ if $v_i\in\cI_X$. Therefore conjugation  by $s(v_1')
\cdots s(v_{p-1}')$ sends 
 $\pi_N(\kappa(v_p))$ to  $\pi_N(\kappa(v_p))^{\sgn(r,\cI,p)}=\gamma(\pi_N(\kappa(v_p)))$ by \cref{defn:GammaTechnical},  so 
  \begin{align*}
&\pi_N(\kappa(v_1))[{}^{s(v_1')} \pi_N(\kappa(v_2))][{}^{s(v_1')s(v_2')}\pi_N(\kappa(v_3))] \cdots [{}^{s(v_1')\cdots s(v_{k-1}')}\pi_N(\kappa(v_k))]\\ &= \gamma(\pi_N(\kappa(v_1))
\pi_N(\kappa(v_2))\cdots \pi_N(\kappa(v_k)))\\
&=\gamma(\pi_N(\kappa(v_1\cdots v_k)))\\
&=\gamma(\pi_N(\kappa(r))).
\end{align*} 
We have shown that \cref{eqn:ExpansionM1} becomes 
 \begin{align}\label{eqn:ExpansionM2}
       1_N
        &= \gamma(\pi_N(\kappa(r)))\tildef_k(v_1',\dots, v_k')
    \end{align}

Let 
$\sigma=\pi_N\circ\kappa\circ \zeta^{-1}$ where $\zeta$ is the bijection as in \cref{defn:PRESEQNB}.
Then $\sigma\colon \bbX\cup\bbY \to N$ 
is the map 
    \[\begin{array}{lllll}
        \sigma(X_i) = \pi_N(\kappa(\zeta^{-1}(X_i))) = \pi_N(\kappa(x_i)), & \quad\quad &  \sigma(X_i^{-1}) = \pi_N(\kappa(x_i))^{-1}, \\
        \sigma(Y_i) = \pi_N(\kappa(\zeta^{-1}(Y_i))) = \pi_N(\kappa(y_i)),&&
        \sigma(Y_i^{-1}) = \pi_N(\kappa(y_i))^{-1}.
    \end{array}\]

For $i\in[1,m]$ the equation $u_i$ is obtained from   $r_i=v_{i,1}\dots v_{i,\abs{r_i}}$ with $v_{i,j}\in \cX\cup\cY\cup\cX^{-1}\cup\cY^{-1}$
 of the form 
 \begin{align*}u_i &=\zeta(\gamma(r_i)) \tildef_{\abs{r_i}}(\tau(v_{i,1}),\dots,\tau(v_{i,\abs{r_i}}))\\
 &=\gamma(\zeta(r_i)) \tildef_{\abs{r_i}}(\tau(v_{i,1}),\dots,\tau(v_{i,\abs{r_i}}))
 \end{align*}
so   
    \begin{align*}
        \sigma(u_i)&= \gamma(\pi_N(\kappa(r_i))) \tildef_{\abs{r_i}}(\tau(v_{i,1}),\dots,\tau(v_{i,\abs{r_i}}))
        =1_N
    \end{align*}
by
 \cref{eqn:ExpansionM2}, which means $\sigma$ is a solution to the system.

By item~(c'), for all $ n \in N$ there exists $ w\in (\cY\cup\cY^{-1})^*$ such that $\kappa(w)= ns(1_Q)$, so $\pi_N(\kappa(w))=n$.
Then $\zeta(w)\in \bbY^*$ satisfies \begin{align*}
    \sigma(\zeta(w)) &= \pi_N(\kappa(\zeta^{-1}(\zeta(w)))\\
    &=\pi_N(\kappa(w))=n
\end{align*} which means  $\gen{\sigma(Y_1), \dots,\sigma(Y_\ell)} = N$, so   \equationsProbNEW returns `Yes'.

    Conversely, assume that 
 there exists an epimorphism $\tau\colon G \to Q$ 
     and for some \QtauPres\ $\Gen{\cX\cup\cY}{\cR}$ for $G$ there is a solution $\sigma\colon \bbX\cup\bbY\to N$ 
     to the system $\PRESEQNB(\tau, \cX,\cY,\cR,\cI, f_s)$
 such that $\gen{\sigma(Y_1),\dots,\sigma(Y_\ell)} = N$.
We will show that there is a homomorphism $\kappa\colon G\to H$ such that $\tau,\kappa$ satisfy conditions (b) and (c') of \cref{lem:epi_iff_NbyQ_epi} and \cref{rmk:itemc'}, thereby proving the existence of an epimorphism from $G$ to $H$.

Let $\kappa\colon (\cX\cup\cY\cup\cX^{-1}\cup\cY^{-1})^* \to H$ be the monoid homomorphism induced by the set map
\begin{align*}
    \kappa(a)=\sigma(\zeta(a))s(\tau(a))
\end{align*} for $a\in \cX\cup\cY\cup\cX^{-1}\cup\cY^{-1}  $. 

Then for any $w=v_1\dots v_n$ where $v_i\in \cX\cup\cY\cup\cX^{-1}\cup\cY^{-1}  $
we have 
\begin{align*}
\kappa(w) &= \kappa(v_1)\cdots \kappa(v_n)\\
&=\sigma(\zeta(v_1))s(\tau(v_1))\cdots \sigma(\zeta(v_n))s(\tau(v_n))\\
&=\sigma(\zeta(v_1))s(v_1')\cdots \sigma(\zeta(v_n))s(v_n')
\end{align*}
where $v_i'=\tau(v_i)$ as before. Inserting $s(v_1')\cdots s(v_j')s(v_j')^{-1}\cdots s(v_1')^{-1}$  we obtain
\begin{align*}
\kappa(w) 
&=\sigma(\zeta(v_1)){}^{s(v_1')}\sigma(\zeta(v_2))
{}^{s(v_1')s(v_2')}\sigma(\zeta(v_3))\cdots {}^{s(v_1')\cdots s(v_{n-1}')}\sigma(\zeta(v_n))s(v_1')\cdots s(v_n')\\
&=\gamma(\sigma(\zeta(v_1))\cdots \sigma(\zeta(v_k))) s(v_1')\cdots s(v_n')\\
&=\gamma(\sigma(\zeta(w)))\tildef_k(v_1',\dots, v_k') s(\tau(v_1\cdots v_k))\ \quad \textnormal{by \cref{eqn:sProductFtilde}}\\
&=\gamma(\sigma(\zeta(w)))\tildef_k(v_1',\dots, v_k') s(\tau(w)).
\end{align*}

If $w\in \mathcal R$ then $\tau(w)=1_Q$ so $s(\tau(w))=1_N$
and $\gamma(\zeta(w))\tildef_k(v_1',\dots, v_k')=
\zeta(\gamma(w))\tildef_k(v_1',\dots, v_k')$ is an equation in the system  $\PRESEQNB(\tau, \cX,\cY,\cR,\cI,f_s)$ so applying $\sigma$ 
we have  $\kappa(w)=1_N$, so by \cref{lem:vonD} $\kappa$ is a homomorphism.

For $g\in G$ suppose $w\in (\cX\cup\cY\cup\cX^{-1}\cup\cY^{-1}  )^*$ spells $g$, then \begin{align*}
\kappa(g)=\kappa(w)
&=\gamma(\sigma(\zeta(w)))\tildef_k(v_1',\dots, v_k') s(\tau(w))\\
&=\gamma(\sigma(\zeta(w)))\tildef_k(v_1',\dots, v_k') s(\tau(g))\\
&=ns(\tau(g))
\end{align*}
where 
$n=\gamma(\sigma(\zeta(w)))\tildef_k(v_1',\dots, v_k') \in N$  
so condition (b) 
of \cref{lem:epi_iff_NbyQ_epi}  is satisfied.

  Since $\gen{\sigma(Y_1), \dots, \sigma(Y_\ell)} = N$, for all $n \in N$ there exists $w \in \bbY^\ast$  such that $\sigma(w) = n$. Then for all $n \in N$ there exists $\zeta^{-1}(w) \in (\cY\cup\cY^{-1})^\ast$ such that \begin{align*}
      \kappa(\zeta^{-1}(w))&=\sigma(\zeta(\zeta^{-1}(w)))s(\tau(\zeta^{-1}(w)))\\
      &=\sigma(w)s(\tau(\zeta^{-1}(w)))\\
      &=ns(\tau(\zeta^{-1}(w)))
  \end{align*}and $\tau(\zeta^{-1}(w))=1_Q$ because $\ker(\tau)=\gen{\cY}$, so $ \kappa(\zeta^{-1}(w))=ns(1_Q)$
  giving condition~(c') of \cref{rmk:itemc'}, thus showing there exists an epimorphism from $G$ to $H$.
\end{proof}

For the rest of this section we  assume  $N$ is an infinite cyclic group (so $H$ is virtually cyclic). 

Recall from \cref{defn:EQN_TO_MAT} that $\EQNMAT(1,t,\ell,(u_i)_{[1,m]},(\MFc_i)_{[1,m]})$ is a triple $(A,b,\ell)$ where  $A\in \Z^{m\times (t+\ell)}$ and $b\in \Z^{m\times 1}$. 

\begin{lemma}\label{lem:EQN_TO_EQMATBSUBSPANZ}
    Let \be\item 
    $N$ be an infinite cyclic group $\langle x\rangle$
    \item $\bbX = \{X_1,X_1^{-1},\dots,X_t,X_t^{-1}\}$, $\bbY = \{Y_1,Y_1^{-1},\dots,Y_\ell,Y_\ell^{-1}\}$ \item $(u_i)_{[1,m]}$ be a system of equations over $N$ where each equation is of the form $u_i = v_i\MFc_i$ with $v_i \in (\bbX\cup\bbY)^\ast$ and $\MFc_i  = x^{b_i}\in N$ is a constant where $b_i\in \Z$.\ee The following are equivalent.
    \be\item \equationsProbNEW returns `Yes' on input $N$ and $(u_i)_{[1,m]}$ 
    \item \MatrixSubspanB returns `Yes' on input $(A,b,\ell) = \EQNMAT(1,t,\ell,(u_i)_{[1,m]},(\MFc_i)_{[1,m]})$.\ee
\end{lemma}

\begin{proof} Suppose 
\equationsProbNEW returns `Yes' on input $N$ and $(u_i)_{[1,m]}$. Then   there exists a solution  $\sigma\colon \bbX\cup\bbY \to N$ given by  \begin{equation*}\label{eqn:ME-sigma-B}
        \sigma\colon \begin{cases}
            X_j \mapsto x^{f_{j}} & j \in [1,t] \\
            Y_k \mapsto x^{f_{t+k}}  & k \in [1,\ell]
        \end{cases}
    \end{equation*}
to $(u_i)_{[1,m]}$ (and to $(\CNF(u_i))_{[1,m]}$) with $\gen{\sigma(Y_1),\dots,\sigma(Y_\ell)} = N$.
Then for $i \in [1,m]$   \begin{equation*}\label{eqn:ME-sum-B}
        \sigma(u_i) = x^{\sum_{j = 1}^t f_{j}\alpha_{(i,j)} + \sum_{k = 1}^\ell f_{t+k}\beta_{(i,k)}}  x^{b_i} = 1_N
\end{equation*}
where  
        \[\alpha_{(i,j)}=\abs{u_i}_{X_j}-\abs{u_i}_{X_j^{-1}}\ \ \textnormal{and}\ \  \beta_{(i,k)}=\abs{u_i}_{Y_k}-\abs{u_i}_{Y_k^{-1}}\] for $j\in[1,t], k\in[1,\ell]$, which holds if and only if
    \begin{equation}\label{eqn:VCeqnSum=0}
        \sum_{j = 1}^t f_{j}\alpha_{(i,j)} + \sum_{k = 1}^\ell f_{t+k}\beta_{(i,k)} + b_i = 0.
     \end{equation}

Recall from \cref{defn:EQN_TO_MAT} that 
    \begin{equation*}
        A = \begin{pmatrix}
            \alpha_{(1,1)} & \cdots & \alpha_{(1,t)} & \beta_{(1,1)} & \cdots & \beta_{(1,\ell)} \\
            \vdots & \ddots & \vdots & \vdots & \ddots & \vdots \\
            \alpha_{(m,1)} & \cdots & \alpha_{(m,t)} & \beta_{(m,1)} & \cdots & \beta_{(m,\ell)} 
        \end{pmatrix}, \ b= \begin{pmatrix}
            b_1 \\ \vdots \\ b_m
        \end{pmatrix}
    \end{equation*}
    and let $\nu \in \Z^{t+\ell}$ be the integer $(t+\ell)$-vector $\nu = (f_1 \ f_2 \ \cdots \ f_{t+\ell})^T$. Then    
    \begin{align}\label{ME:calculation}
        A\nu +b&= \begin{pmatrix}
            \alpha_{(1,1)} & \cdots & \alpha_{(1,t)} & \beta_{(1,1)} & \cdots & \beta_{(1,\ell)} \\
            \vdots & \ddots & \vdots & \vdots & \ddots & \vdots \\
            \alpha_{(m,1)} & \cdots & \alpha_{(m,t)} & \beta_{(m,1)} & \cdots & \beta_{(m,\ell)} 
        \end{pmatrix} \begin{pmatrix}
        f_1 \\ \vdots \\ f_{t+\ell}
    \end{pmatrix} + \begin{pmatrix}
        b_1 \\ \vdots \\ b_{t+\ell}
    \end{pmatrix}
    \nonumber\\
    &= \begin{pmatrix}
        \sum_{j=1}^t f_{j}\alpha_{(1,j)} + \sum_{j=1}^\ell f_{t+j}\beta_{(1,j)} + b_1 \\
        \vdots  \\
        \sum_{j=1}^t f_{j}\alpha_{(m,j)} + \sum_{j=1}^\ell f_{t+j}\beta_{(m,j)} + b_m
        \end{pmatrix}\\
    &= \begin{pmatrix}
        0  \\
        \vdots  \\
        0 
        \end{pmatrix} \textnormal{by \cref{eqn:VCeqnSum=0}.}\nonumber
    \end{align}  

   Since
    $\gen{\sigma(Y_1),\dots,\sigma(Y_\ell)} = N$, for all $h \in N$ there exists $w \in \bbY^\ast$ such that $\sigma(w) = h$, so for all $z \in \Z$ there exists $w \in \bbY^\ast$ such that $\varphi(\sigma(w)) = z$
where $\varphi\colon N \to \Z$
 is the natural isomorphism. 
We have \begin{align*}
        z = \varphi(\sigma(w)) 
        &= \varphi(\sigma(Y_1)^{a_1}\cdots \sigma(Y_\ell)^{a_\ell}) \\
        &= a_1f_{t+1} + \cdots + a_\ell f_{t+\ell}
    \end{align*}
    where $a_i = \abs{w}_{Y_i} - \abs{w}_{Y_i^{-1}}$, so $z\in  \spanZ{c_{t+1},\dots,c_n}$, so \MatrixSubspanB returns `Yes' on input $(A,b,\ell) = \EQNMAT(1,t,\ell,(u_i)_{[1,m]},(\MFc_i)_{[1,m]})$.

    Conversely suppose \MatrixSubspanB on input $(A,b,\ell)= \EQNMAT(1,t,\ell,(u_i)_{[1,m]},(\MFc_i)_{[1,m]})$ returns  an 
integer $n$-vector
$\nu$  with $A\nu + b = 0$ and  $\spanZ{(\nu|_\ell)^T}=\Z$.
    Define  $\sigma\colon \bbX\cup\bbY \to N$  by  \begin{equation*}
        \sigma\colon \begin{cases}
            X_j \mapsto x^{\nu_{j}} & j \in [1,t] \\
            Y_k \mapsto x^{\nu_{t+k}}  & k \in [1,\ell]
        \end{cases}
    \end{equation*}
    Since $\spanZ{(\nu|_\ell)^T}=\Z$ then every $z\in \Z$ can be expressed as \[z=a_1\nu_{t+1}+\cdots+a_\ell\nu_{t+\ell}\] for $a_i\in \Z$, so for each $x^z\in N$ there exists $w=Y_1^{a_1}\cdots Y_1^{\nu_{a_\ell}}\in \bbY^*$ so that $\sigma(w)=x^{a_1\nu_{t+1}+\cdots+a_\ell\nu_{t+\ell}}=x^z$, so
 $N\subseteq \gen{\sigma(Y_1),\dots,\sigma(Y_\ell)}$.
Since $A\nu + b = 0$, by \cref{defn:EQN_TO_MAT} we have $\sigma(u_i)=1_N$ for $i\in[1,m]$ by the calculation in \cref{ME:calculation}, so \equationsProbNEW returns `Yes'.
\end{proof}

Combining the above results with the fact that will be proved in \cref{sec:LA} that \MatrixSubspanB can be decided in polynomial time gives the following.

\begin{theorem}\label{thm:vC_NP}
    $\Epi{\Arb}{\vC}$ is in \NP. 
\end{theorem}

\begin{proof}
    Let $G \in \Arb$ be given by a finite presentation $\Gen{\cG}{\cR}$ and $H \in \PMonlyExt$ a virtually cyclic group given by $N = \gen{x}$ an infinite cyclic,   a multiplication table for $Q \in \Fin$, and  special extension data $(\cI,f_s)$.
    Using 
    \cref{lem:epi_iff_equations_pmonly}
    we may verify the existence of an epimorphism from $G$ to $H$ by verifying that
 there exists an epimorphism $\tau\colon G \to H$, and for some \QtauPres\ $\Gen{\cX\cup\cY}{\cR}$ for $G$, \equationsProbNEW returns `Yes' on input $N$ and $\PRESEQNB(\tau, \cX,\cY,\cR,\cI, f_s)$.

     On input $G, H$ as above:
    \begin{enumerate}
        \item guess and verify that the set map $\tau\colon \cG \to Q$ extends to an epimorphism $\tau\colon G \to Q$ (this is the only non-deterministic step of the algorithm)
        \item construct a \QtauPres\ $\Gen{\cX\cup\cY}{\cR}$, and set 
        $\cI_X=\{x\in \cX\mid \tau(x)\in\cI\}$
        \item construct a system of equations  $\PRESEQNB(\tau, \cX,\cY,\cR,\cI, f_s)$ denoted $(u_i)_{[1,\abs{\cR}]}$ where   $v_i$ is an equation without constants,  $\MFc_i \in N$ is a constant, and $u_i = v_i\MFc_i$
        \item construct the triple $(A,b,\abs{\cY}) = \EQNMAT(1, (u_i)_{[1,\abs{\cR}]}, (\MFc_i)_{[1,\abs{\cR}]})$ 
        \item return the solution to \equationsProbNEW on input $(A,b,\abs{\cY})$. 
    \end{enumerate}

The correctness of this algorithm follows from \cref{lem:EQN_TO_EQMATBSUBSPANZ,lem:epi_iff_equations_pmonly}.
    The time complexity is as follows:
    \begin{enumerate}
        \item verifying in polynomial time that $\tau$ is an epimorphism follows from  \cref{lem:epi_finite}
        \item constructing a \QtauPres\ is in \P by \cref{lem:buildQtau}, and $\cI_\cX$ is immediate from the input
                \item constructing $\PRESEQNB(\tau, \cX,\cY,\cR,\cI, f_s)$ is in \P\  immediately from the  definition
     
           \item constructing $\EQNMAT(1,(u_i)_{[1,\abs{\cR}]},(\MFc_i)_{[1,\abs{\cR}]})$ is in \P\ by \cref{lem:BuildInP}

        \item \MatrixSubspanB is solved in \P by \cref{prop:EQMATBSUBSPANZ_P} (Subsection~\ref{subsec:EQMATBSUBSPANZ}).
    \end{enumerate}
        Thus, our algorithm is in \NP. 
    \end{proof}

\section{Inverse restricted semi-direct targets}\label{sec:semipmOne}

Using results from the previous two sections, we are able to 
extend  the class of virtually abelian targets for which epimorphism from a finitely presented group is decidable, as follows.

Recall that $\PMonlySemi$ is the class of  $N$ by $Q$ extensions such that $Q$ is finite, $N$ is abelian,  there exists a transversal map $s$  and a subset $\cI \subseteq Q$ such that $f_s= \constMap$ and 
for all $n\in N$,  
\[{}^{s(q)}n=
\begin{cases}n^{-1} & q\in \cI\\
 n & q\in Q\setminus \cI
\end{cases}\]

\begin{theorem}\label{thm:restrictedPM_NP} Let $\cT$ be a subclass of $\PMonlySemi$. Then 
    $\Epi{\Arb}{\cT}$ is in \NP.
\end{theorem}

\begin{proof}  Let $G \in \Arb$ be given by a finite presentation $\Gen{\cG}{\cR}$ and $H \in \PMonlySemi$  given by an integer $d\in \N$ encoding $N \in \FreeAb$ of rank $d$,   a multiplication table for $Q \in \Fin$, and  special extension data $(\cI,\constMap)$.

    Since $\PMonlySemi$ is a subclass of $\PMonlyExt$, by \cref{lem:epi_iff_equations_pmonly} we may verify the existence of an epimorphism from $G$ to $H$ by verifying that:
    \begin{enumerate}
        \item[(i)] there exists an epimorphism $\tau\colon G \to H$ 
        \item[(ii)] for some \QtauPres\ $\Gen{\cX\cup\cY}{\cR}$ for $G$,  \equationsProbNEW returns `Yes' on input $N$ and $\PRESEQNB(\tau,\cX,\cY,\cR, \cI,\constMap)$.
    \end{enumerate}
Note that 
    since $f_s=\constMap$,  $\PRESEQNB(\cX,\cY,\cR, \cI_\cX, \constMap)$ is a system of equations without constants.

    The following procedure solves our problem. On input as above, 
    \begin{enumerate}
        \item guess a set map $\tau\colon \cG \to Q$ and verify it extends to an epimorphism $\tau\colon G \to Q$
        \item construct a \QtauPres\ $\Gen{\cX\cup\cY}{\cR}$, and set $\cI_X=\{x\in \cX\mid \tau(x)\in\cI\}$
        \item construct the system of equations without constants $\PRESEQNB(\tau, \cX,\cY,\cR, \cI,\constMap)$ denoted  $(u_i)_{[1,m]}$
        \item return `Yes'  \MatrixSubspanA on input $(A,0,\abs{\cY}) = \EQNMAT(d,(u_i)_{[1,m]},(1_N)_{[1,m]})$ returns `Yes', and `No' otherwise.
        
    \end{enumerate}

  The correctness of the procedure follows from  \cref{NewLemma:EQN_IFF_EQMATSUBSPAN,lem:epi_iff_equations_pmonly}.
   The time complexity is as follows.
    \begin{enumerate}
        \item Step~(1) is in \NP by  \cref{lem:epi_finite}; this is the only non-deterministic step of our algorithm. 
        \item We can construct a \QtauPres in \P  by \cref{lem:buildQtau}, and $\cI_\cX$ is immediate.
        \item Constructing $\PRESEQNB(\tau,\cX,\cY,\cR,\cI,\constMap)$ in \P is clear from its definition.
        \item Constructing $\EQNMAT(d,(u_i)_{[1,m]},(1_N)_{[1,m]})$ is in \P by \cref{lem:BuildInP}, and \MatrixSubspanA is solved \P by \cref{prop:EQMATSUBSPAN_P} (see Subsection~\ref{subsec:EQMATSUBSPAN}).
    \end{enumerate}
    Thus, our algorithm is in \NP.
\end{proof}

\section{Proving \MatrixSubspanA and \MatrixSubspanB are in \P}\label{sec:LA}

In this section we prove that the  integer matrix problems \MatrixSubspanA and \MatrixSubspanB can be decided in polynomial time. 
Recall that throughout this paper, integer matrices are assumed to be given with entries in binary.
Throughout this section we let $R$ denote either $\Z$ or $\Z_p$ for some prime $p$ (the ring of integers mod $p$).

\begin{definition}[Smith normal form and $\onecount$]\label{def:SNF}
    Let $A \in R^{m\times n}$. We call 
    $(K,D,L)$ 
    a \emph{Smith normal form} (SNF) for $A$ if  $A=KDL$, 
   $K \in \GL(m,R)$,  $L \in \GL(n,R)$,  and $D\in R^{m\times n}$ has the form 
    \begin{equation*}
        D = \begin{pmatrix}
            \begin{array}{c|c}
                M & 0 \\
                \hline
                0 & 0
            \end{array} 
            \end{pmatrix}
    \end{equation*}
    where $M$ is a diagonal matrix of the form
    \begin{equation*}
        M = \begin{pmatrix}
            \MFd_1 & 0 & \cdots & 0 \\
            0 & \MFd_2 & \cdots & 0 \\
            \vdots & \vdots & \ddots & \vdots \\
            0 & 0 & \cdots & \MFd_r
        \end{pmatrix}
    \end{equation*}
        for some $0 \leq r \leq \min(m,n)$,  each $\MFd_i \neq 0$ 
        such that $\MFd_i \mid \MFd_{i+1}$ for all $i \in [1,r-1]$. Note $\rank(D)=r$.
        When $R=\Z$ we also require $\MFd_i>0$ for $i\in[1,r]$.

        We denote the number of units (invertible elements) of $R$ on the diagonal of $D$ as $\onecount(D) = \max\{i\mid \MFd_i\in R^*\}$, so if $R=\Z$ then $\onecount(D)$ is the number of $1$s on the diagonal, and if $R=\Z_p$ then $\onecount(D)$ is the number of non-zero entries on the diagonal.
\end{definition}

\begin{lemma}[{\cite[Proposition 3.2]{sims1994}}] If  $R=\Z$, then 
the matrix $D$  in \cref{def:SNF} is unique. 
\end{lemma}

 Moreover for $A\in R^{m\times n}$ if $(K,D,L),(K'D',L')$ are both Smith normal forms for $A$ then $\onecount(D)=\onecount(D')$.
 Note that in many papers the 
Smith normal form for $A\in\Z^{m\times n}$ is defined just to be the matrix $D$.

If $A\in \Z_p^{m\times n}$ then (since $\Z_p$ is a field) multiplying left and right by elementary matrices one may easily  obtain a Smith normal form for $A$ where diagonal entries of $D$ are either $1$ or $0$. For  $A \in \Z^{m \times n}$ the problem requires more attention.

\begin{theorem}[Computing SNF; \cite{kannan1979}]
    \label{thm:snf_poly} The following computational problem is in \P. 

    \compalgo{a matrix $A\in 
\Z^{m\times n}$}{compute $K \in \GL(m,\Z), L \in \GL(n,\Z), D\in \Z^{m\times n}$ such that $(K,D,L)$ is a SNF for $A$.}
\end{theorem}

We also  observe the following facts. The first is straightforward, and the second can be found in \cite{kannan1979}.
\begin{lemma}\label{lem:poly_processes}
    The following  
    calculations
    can be achieved in polynomial time.
    \be
        \item Given $a_1,\dots,a_s \in \Z$, calculate $\gcd(a_1,\dots,a_s)$.
        \item Given $A \in \GL(n,\Z)$, calculate $A^{-1}$.
  \ee
\end{lemma}

Putting the above results together, we have
\begin{lemma}\label{lem:Solve_Ax=b}
    The following algorithmic problem can be answered in polynomial time.
    On input $A \in \Z^{m \times n}$ and $b \in \Z^m$, \be\item decide if there  exists $x \in \Z^n$ such that $Ax + b = 0$ (returning `Yes/No')
    \item if `Yes', return $u_1, \dots, u_k \in \Z^n$ and $c \in \Z^n$ such that for $x \in \Z^n$, $Ax + b = 0$ if and only if $x \in \spanX{u_1,\dots,u_k}{c}$.
\ee
\end{lemma}

\begin{proof}
The following procedure solves our problem.
    \begin{enumerate}
        \item Calculate an SNF  $(K,D,L)$ of $A$ and set  $r=\rank(D)$.
        \item Let  $\MFb = -K^{-1}b$ with $i$-th entry $\MFb_i$ and let $\MFd_i$ be the $i$-th non-zero diagonal entry of $D$. If  $\MFb_i/\MFd_i \not \in \Z$ for some $i \in [1,r]$ then output `No'. 
        If $\MFb_i \neq 0$ for some  $i \in [r+1,m]$   then output `No'.  Else return `Yes'. (Thus we may  assume from here that $\MFb_i/\MFd_i \in \Z$ for each $i\in[1,r]$ and $\MFb_i = 0$ for each $i \in [r+1,m]$.)
        \item Denote the $(i,j)$-th element of $L^{-1}$ as $l_{(i,j)}$,  $u_i$ the $(r+i)$-th column of $L^{-1}$ for $i \in [1,n-r]$, and $c_i = \sum_{j=1}^r l_{(i,j)}\MFb_j/\MFd_j$ for $i \in [1,n]$. Set $k = n-r$ and $c = (c_1 \ \cdots \ c_n)^T$, and  output $u_1,\dots,u_k \in \Z^n$ and $c \in \Z^n$.
    \end{enumerate}
    Step~(1) takes polynomial time by \cref{thm:snf_poly}, steps~(2) and~(3) require the inverse of an integer matrix which can be obtained in polynomial time 
    by \cref{lem:poly_processes}, and basic calculations with integers written in binary which are polynomial time.
    Thus, our process takes polynomial time.

    Now, we justify the correctness of the procedure. We wish to solve the  equation $  Ax - b=0$ where $A=KDL$ 
    for $x\in \Z^n$, so we wish to solve 
    \begin{align*}
        DLx - K^{-1}b&=0.
    \end{align*}
    Let $y = Lx \in \Z^n$ and $-K^{-1}b = \MFb \in \Z^m$, so our equation becomes
    \begin{align}\label{eqn:Dy=b}
        Dy =
        \begin{pmatrix}
            \mathfrak{d}_1 & \cdots & 0 & \cdots & 0 \\
            \vdots & \ddots & \vdots & & \vdots \\
            0 & \cdots & \mathfrak{d}_r & \cdots & 0 \\
            \vdots & & \vdots & \ddots & \vdots\\
            0 & \cdots & 0 & \cdots & 0  
        \end{pmatrix} \begin{pmatrix}
            y_1 \\ \vdots \\ y_r \\ \vdots \\ y_n
        \end{pmatrix}&= \begin{pmatrix}
            \MFb_1 \\ \vdots \\ \MFb_r \\ \vdots \\ \MFb_m
        \end{pmatrix} = \MFb.
    \end{align}
    From \cref{eqn:Dy=b} it is clear that 
     for a solution to exist we need $y_i = \MFb_i/\MFd_i$ for $i \in [1,r]$ and $\MFb_{r+1},\dots,\MFb_m = 0$, and since $y\in\Z^n$ we have the condition to return `Yes/No' in step (2).
If `Yes', let $a_i = \MFb_i / \MFd_i\in\Z$ for $i\in[1,r]$, so 
      \[\MFb=(   a_1\mathfrak{d}_1  \ \cdots\ a_r\mathfrak{d}_r \ 0 \ \cdots \ 0)^T
    \]
    and in  such case a solution has the form 
\[ y=( a_1  \ \cdots \ a_r \ t_{r+1} \ \cdots \ t_n)^T\]
    for any $t_i\in \Z, i> r$. 

    Recall that we write $l_{(i,j)}$ for the $(i,j)$-th entry of $L^{-1}$.
    Since $Lx = y$ we have
    \begin{align*}        
        x &= L^{-1}y \\
        &= \begin{pmatrix}
            l_{(1,1)} & \cdots & l_{(1,n)} \\
            \vdots & \ddots & \vdots \\
            l_{(n,1)} & \cdots & l_{(n,n)}
        \end{pmatrix}
        \begin{pmatrix}
            a_1 \\ \vdots \\a_r \\ t_{r+1} \\ \vdots \\ t_n
        \end{pmatrix} \\
        &= \begin{pmatrix}
            l_{(1,1)}a_1 + \cdots + l_{(1,r)}a_r + l_{(1,r+1)}t_{r+1} + \cdots + l_{(1,n)}t_{n} \\
            \vdots \\
            l_{(n,1)}a_1 + \cdots + l_{(n,r)}a_r + l_{(n,r+1)}t_{r+1} + \cdots + l_{(n,n)}t_{n}
        \end{pmatrix}.
    \end{align*} 
    Let $c_i 
    = \sum_{j=1}^r l_{(i,j)}a_j$ for $i \in [1,n]$, then
    \begin{equation*}
        x = \begin{pmatrix}
            c_1 \\ \vdots \\ c_n
        \end{pmatrix}
        + t_{r+1} \begin{pmatrix}
            l_{(1,r+1)} \\ \vdots \\ l_{(n,r+1)}
        \end{pmatrix} + \cdots
        + t_{n} \begin{pmatrix}
            l_{(1,n)} \\ \vdots \\ l_{(n,n)}
        \end{pmatrix}.
    \end{equation*}
    Setting $u_i$ to be the $(n-r+i)$-th column of $L^{-1}$ for $i \in [1,k]$ we have shown that for $x \in \Z^n$, $Ax -b=0$ if and only if $x \in \spanX{u_1,\dots,u_k}{c}$.
    \end{proof}
    
\begin{lemma}\label{lem:remove_rows_span}
    Let $U \in \Z^{n \times k}$, $b \in \Z^n$, $\ell \in \Z$, and   $c = (b_{n-\ell+1} \ \cdots \ b_n)^T$.
    Then the following are equivalent.
    \be\item There exists a
    matrix $V\in\Z^{n\times d}$ such that $\spanZ{(V|_\ell)^T} = \Z^d$ and each column of $V$  lies in $\spanX{U}{b}$ 
    \item  There exists a  matrix $W\in\Z^{\ell\times d}$ such that  $\spanZ{W^T} = \Z^d$ and
    each column of $W$ lies in  $\spanX{U|_\ell}{c}$. \ee
\end{lemma}

\begin{proof}
    Assume there exists a
    matrix $V=\begin{pmatrix}
         v_1 & \cdots & v_d
      \end{pmatrix}\in\Z^{n\times d}$ such that $\spanZ{(V|_\ell)^T} = \Z^d$ and each column $v_i$ of $V$  lies in $\spanX{U}{b}$. 
    For each $i\in[1,d]$  let $w_i \in \Z^\ell$ be the last $\ell$ entries of $v_i$ and set $W=\begin{pmatrix}
         w_1 & \cdots & w_d
      \end{pmatrix}\in\Z^{\ell\times d}$, so $W = V|_\ell$.
    Since $v_i \in \spanX{U}{b}$, $w_i \in \spanX{U|_\ell}{c}$, and so   $\spanZ{W^T} = \spanZ{(V|_\ell)^T} = \Z^d$.

    Conversely, assume there exist 
    a matrix $W=\begin{pmatrix}
         w_1 & \cdots & w_d
      \end{pmatrix}\in\Z^{\ell\times d}$ such that  $\spanZ{W^T} = \Z^d$ and
    each column $w_i$ of $W$ lies in  $\spanX{U|_\ell}{c}$.
    For each $i\in[1,d]$ there exist $\alpha_{i,j}\in\Z$ so that \[w_i=c+\alpha_{i,1}\tilde u_1+\cdots +\alpha_{i,k}\tilde u_k\] where $\tilde u_j\in\Z^\ell$ are the columns of $U|_\ell$.
    Define $v_i\in \Z^k$ to be 
    \[v_i=b+\alpha_{i,1}u_1+\cdots +\alpha_{i,k}u_k\] 
    where $u_j\in\Z^n$ are the columns of $U$.
   Then the matrix $V=\begin{pmatrix}
         v_1 & \cdots & v_d
     \end{pmatrix}$ satisfies $V|_\ell=W$ so 
   $\spanZ{(V|_\ell)^T} = \spanZ{W^T} = \Z^d$, and each column of $V$ is in $\spanX{U}{b}$ by construction.
    \end{proof}

\subsection{Solving \MatrixSubspanA } 
\label{subsec:EQMATSUBSPAN}
In this subsection, we show that \MatrixSubspanA can be decided in polynomial time. We start with two simple observations. Recall that $R=\Z$ or $\Z_p$.

\begin{lemma}\label{lem:spanA=spanB}
    Let $A,B \in R^{m\times n}$ and $L \in \GL(n,R)$. If $A = BL$ then $\spanZ{A} = \spanZ{B}$.
\end{lemma}

\begin{proof}
By definition, $\spanZ{M}$ is the image of the function $R^n\to R^m$ defined by $v\mapsto Mv$. Since $L\in  \GL(n,R)$,  $v\mapsto Lv$ is a bijection of $R^n$, so 
  $\spanZ{BL}=\spanZ{B}$.  \end{proof}

\begin{lemma}\label{lem:GLnZcols_generate_basis}
    Let $K \in \GL(\ell,\Z)$ and denote the $(i,j)$-th element as $k_{i,j}$. Then for each $j\in[1,\ell]$ there exists $a_1,\dots,a_\ell \in \Z$ such that\[
    \begin{array}{l}
        a_1k_{1,j} + \cdots + a_\ell k_{\ell,j} = 1
 \textnormal{ and}\\
        a_1k_{1,s} + \cdots + a_\ell k_{n,s} = 0
        \textnormal{ for $s \in [1,\ell]$ and $s \neq j$}.
    \end{array}\]
\end{lemma}

\begin{proof}
    Denote the $(i,j)$-th element of $K^{-1}$ as $c_{i,j}$. Then 
    \begin{align*}
        K^{-1}K &= \begin{pmatrix}
            c_{1,1} & \cdots & c_{1,\ell} \\
            \vdots & \ddots & \vdots \\
            c_{\ell,1} & \cdots & c_{\ell,\ell}
        \end{pmatrix} \begin{pmatrix}
            k_{1,1} & \cdots & k_{1,\ell} \\
            \vdots & \ddots & \vdots \\
            k_{\ell,1} & \cdots & k_{\ell,\ell}
        \end{pmatrix} \\
        &= \begin{pmatrix}
            \sum_{i=1}^\ell c_{1,i}k_{i,1} & \cdots & \sum_{i=1}^\ell c_{1,i}k_{i,\ell} \\
            \vdots & \ddots & \vdots \\
            \sum_{i=1}^\ell c_{\ell,i}k_{i,1} & \cdots & \sum_{i=1}^\ell c_{\ell,i}k_{i,\ell}
        \end{pmatrix} 
        = \begin{pmatrix}
            1 & \cdots & 0 \\
            \vdots & \ddots & \vdots \\
            0 & \cdots & 1
        \end{pmatrix}
    \end{align*}
  so for each $j\in[1,\ell]$ we have \begin{align*}
        c_{j,1}k_{1,j} + \cdots + c_{j,\ell}k_{\ell,j} &= 1 \\
        c_{s,1}k_{1,s} + \cdots + c_{s,\ell}k_{\ell,s} &= 0 \textnormal{ for $s \in [1,\ell]$ and $s \neq j$} 
    \end{align*}
    The result follows for fixed $j\in[1,\ell]$ by setting      
    $a_1 = c_{j,1}, \dots, a_\ell = c_{j,\ell}$.
\end{proof}

Using these facts, we obtain the following.

\begin{lemma}\label{lem:ifOneCount_then_existv}
    Let $A \in \Z^{\ell \times n}$ 
    with  SNF $(K,D,L)$, and  $d \in \N_{+}$.
    If $\onecount(D) \geq d$ then there exists a matrix $V\in\Z^{\ell\times d} $  such that $\spanZ{V^T} = \Z^d$ and each column of $V$  lies in $\spanZ{A}$.
\end{lemma}

\begin{proof}
    Since $A = KDL$, by \cref{lem:spanA=spanB} we have $\spanZ{A} = \spanZ{KD}$.  Since the first $d$ entries along the diagonal of $D$ are $1$'s,  the first $d$ columns of $K$ are in $\spanZ{KD}$. Let $v_1,\dots,v_d \in \Z^\ell$ be the first $d$ columns of $K$, so $v_i\in  \spanZ{KD}=\spanZ{A}$, and let $V = (v_1 \ \cdots \ v_d)$.

    Denote the elements of $K$ as $k_{i,j}$, so 
    $v_j = (k_{1,j} \ \cdots \ k_{\ell,j})^T$ for  $j\in[1,d]$.
    By \cref{lem:GLnZcols_generate_basis}, for each $j \in [1,\ell]$ there exists $a_{1}, \dots, a_{\ell}\in\Z$ such that
    \begin{align*}
        a_{1} k_{1,j} + \cdots + a_{\ell} k_{\ell,j} &= 1 \\
        a_{1} k_{1,s} + \cdots + a_{\ell} k_{\ell,s} &= 0 \textnormal{ for $s \in [1,\ell]$ and $s \neq j$},
    \end{align*} that is, 
    \begin{align*}
        a_{1} \begin{pmatrix} k_{1,1}\\\vdots\\k_{1,j} \\\vdots \\k_{1,d}  \end{pmatrix}
        +\cdots +
        a_{\ell} \begin{pmatrix} k_{\ell,1}\\\vdots\\
        k_{\ell,j} \\
        \vdots\\ k_{\ell,d}  \end{pmatrix}
        =
        \begin{pmatrix} 0\\\vdots\\
        1 \\
        \vdots\\ 0  \end{pmatrix}=e_j
    \end{align*}
where the vectors $(k_{i,1},\dots, k_{i,d})^T$ are the columns of $V^T$. Thus, we have shown that $e_j\in\spanZ{V^T}$ for each $j\in [1,d]$, so 
$\spanZ{V^T}=\Z^d$.
\end{proof}

If $A\in\Z^{m\times n}$, let $[A]_p\in \Z_p^{m\times n}$ denote the matrix $(a_{i,j} \mod p)_{i\in[1,m],j\in[1,n]}$ where $a_{i,j}$ is the $i,j$-th entry of $A$.  
For $B \in \Z_p^{m\times n}$, $\spanZp{B}$ is the set of all $\Z_p$ linear combinations of the columns of $B$. 

\begin{lemma}\label{lem:Rmodule_onecount}
    Let $A \in R^{m\times n}$ with SNF $(K,D,L)$ such that $\rank(D) = \onecount(D)$. If there exists $V \in R^{m\times d}$ such that $\spanZ{V^T} = R^d$ and the columns of $V$ lie in $\spanZ{A}$ then $\onecount(D) \geq d$. 
\end{lemma}

\begin{proof} 
    Let $\onecount(D) = c$, observe that when $R = \Z$ we have $KD$ as the $m\times d$ matrix whose first $c$ columns are the first $c$ columns of $K$, and remaining $d-c$ columns are $0 \in \Z^m$. In the case $R = \Z_p$, as $\Z_p$ is a field, then w.l.o.g assume all non-$0$ diagonals of $D$ are $1$. 

    By \cref{lem:spanA=spanB} $\spanZ{A} = \spanZ{KD}$, so the columns of $V$ lie in $\spanZ{KD}$. Denote $v_i$ as the $i$-th column of $V$ and $k_{i,j}$ as the $i,j$-th element of $K$. As $v_i \in \spanZ{KD}$, for $j \in [1,d]$ there exists $t_{j,1},\dots,t_{j,c} \in R$ such that
    \begin{equation*}
        v_j = t_{j,1}\begin{pmatrix}
            k_{1,1} \\
            \vdots \\
            k_{m,1}
        \end{pmatrix} +
        \cdots + 
        t_{j,c}\begin{pmatrix}
            k_{1,c} \\
            \vdots \\
            k_{m,c}
        \end{pmatrix},
    \end{equation*}
    and so
    \begin{equation*}
        V = \begin{pmatrix}
            \sum_{i=1}^{c} t_{1,i}k_{1,i} & \cdots & \sum_{i=1}^{c} t_{d,i}k_{1,i} \\ 
            \vdots & \ddots & \vdots \\
            \sum_{i=1}^{c} t_{1,i}k_{m,i} & \cdots & \sum_{i=1}^{c} t_{d,i}k_{m,i}
        \end{pmatrix}.
    \end{equation*}
   
    Since $\spanZ{V^T} = R^d$ and $e_1,\dots,e_d \in \spanZ{V^T}$,  for $\ell \in [1,d]$ there exists $\rho_{\ell,1},\dots,\rho_{\ell_,m} \in R$ such that 
    \begin{align*}
        e_\ell &= \rho_{\ell,1}\begin{pmatrix}
            \sum_{i=1}^{c} t_{1,i}k_{1,i} \\ \vdots \\ \sum_{i=1}^{c} t_{d,i}k_{1,i}
        \end{pmatrix} + \cdots + \rho_{\ell,m}\begin{pmatrix}
            \sum_{i=1}^{c} t_{1,i}k_{m,i} \\ \vdots \\ \sum_{i=1}^{c} t_{d,i}k_{m,i}
        \end{pmatrix} \\
        &= \begin{pmatrix}
            \rho_{\ell,1} \sum_{i=1}^{c} t_{1,i}k_{1,i} + \cdots + \rho_{\ell,m}\sum_{i=1}^{c} t_{1,i}k_{m,i} \\
            \vdots \\
            \rho_{\ell,1} \sum_{i=1}^{c} t_{d,i}k_{1,i} + \cdots + \rho_{\ell,m} \sum_{i=1}^{c} t_{d,i}k_{m,i}
        \end{pmatrix} \\
        &= \begin{pmatrix}
            \sum_{i = 1}^m \rho_{\ell,i}  (\sum_{j=1}^c t_{1,j} k_{i,j}) \\
            \vdots \\
            \sum_{i = 1}^m \rho_{\ell,i}  (\sum_{j=1}^c t_{d,j} k_{i,j})
        \end{pmatrix} \\ 
        &= \begin{pmatrix}
            \sum_{j=1}^c t_{1,j} (\sum_{i = 1}^m \rho_{\ell,i}k_{i,j}) \\
            \vdots \\
            \sum_{j=1}^c t_{d,j} (\sum_{i = 1}^m \rho_{\ell,i}k_{i,j})
        \end{pmatrix}.
    \end{align*}
    As the concatenation of $e_1,\dots,e_d$ is the $d\times d$ identity matrix, we have 
    \begin{align*}
        \begin{pmatrix}
            1 & \cdots & 0 \\
            \vdots & \ddots & \vdots \\
            0 & \cdots & 1
        \end{pmatrix} &= \begin{pmatrix}
            \sum_{j=1}^c t_{1,j} (\sum_{i = 1}^m \rho_{1,i}k_{i,j}) & \cdots & \sum_{j=1}^c t_{1,j} (\sum_{i = 1}^m \rho_{d,i}k_{i,j}) \\
            \vdots & \ddots & \vdots \\
            \sum_{j=1}^c t_{d,j} (\sum_{i = 1}^m \rho_{1,i}k_{i,j}) & \cdots & \sum_{j=1}^c t_{d,j} (\sum_{i = 1}^m \rho_{d,i}k_{i,j})
        \end{pmatrix} \\
        &= \begin{pmatrix}
            t_{1,1} & \cdots & t_{1,c} \\
            \vdots & \ddots & \vdots \\
            t_{d,1} & \cdots & t_{d,c}
        \end{pmatrix} \begin{pmatrix}
            \sum_{i = 1}^m \rho_{1,i}k_{i,1} & \cdots & \sum_{i = 1}^m \rho_{d,i}k_{i,1} \\
            \vdots & \ddots & \vdots \\
            \sum_{i = 1}^m \rho_{1,i}k_{i,c} & \cdots & \sum_{i = 1}^m \rho_{d,i}k_{i,c}
        \end{pmatrix}.
    \end{align*}
    Recall from standard linear algebra that for any real-valued matrices $A \in \R^{d \times c}$, $B \in \R^{c\times d}$ we have $\rank(AB)\leq \min\{\rank(A),\rank(B)\}$. Since $\rank(D) = d$  it follows that $c \geq d$.
\end{proof}

\begin{lemma}\label{lem:span_modp_invariant}
    Let $p$ be a prime, $A \in \Z^{m\times n}$. If there exists $V\in \Z^{m \times d}$ such that the columns of $V$ lie in $\spanZ{A}$ and $\spanZ{V^T} = \Z^d$, then there exists $W \in \Z_p^{m  \times d}$ such that the columns of $W$ lie in $\spanZp{[A]_p}$ and $\spanZp{W^T} = \Z_p^d$.
\end{lemma}

\begin{proof}
    Assume there exists $V\in \Z^{m\times d}$ such that the columns of $V$ lie in $\spanZ{A}$ and $\spanZ{V^T} = \Z^d$, then there exist $e_1, \dots, e_d \in \spanZ{V^T} = \Z^d$ and $f_1, \dots, f_d \in \Z^m$ (the standard basis of dimension $m$) $f_1,\dots,f_d \in \spanZ{A}$. So let $V \in \Z^{m \times d}$ be the diagonal matrix with $d$ entries of $1$ on the diagonal then as $W = [V]_p$, that is, each entry $w_{i,j}$ of $W$ is equal to $v_{i,j}\mod p$.
    Then $W^T \in \Z_p^{d \times m}$ is the diagonal matrix with $d$ entries of $1$ on the diagonal and so $\spanZ{W^T} = \Z_p^d$.
\end{proof}

\begin{corollary}\label{corr:VT_span_checkOneCount}
    Let $A \in \Z^{m \times n}$, $\min(m,n) \geq d \in\N_+$, and $(K,D,L)$ is an SNF for $A$.
    If there exists $V \in \Z^{m\times d}$ such that the columns of $V$ lie in $\spanZ{A}$ and $\spanZ{V^T} = \Z^d$, then  $\onecount(D) \geq d$.
\end{corollary}

\begin{proof}
If  $\onecount(D) = \rank(D)$ then \cref{lem:Rmodule_onecount} proves the claim.
Else $\onecount(D) 
        < \rank(D)$.

By \cref{lem:span_modp_invariant} for any prime $p$ there will exist $W \in \Z_p^{m\times d}$ such that the columns of $W$ lie in $\spanZp{[A]_p}$ and $\spanZp{W^T} = \Z_p^d$. Let $c = \onecount(D)$ and $p \mid \MFd_{c+1}$ (the first non-$1$ diagonal entry of $D \in \Z^{m\times n}$) and so $\rank([D]_p) = \onecount([D]_p)$. This gives two cases.
    \begin{enumerate}
        \item If there exists $V \in \Z^{m\times d}$ such that the columns of $V$ lie in $\spanZ{A}$ and $\spanZ{V^T} = \Z^d$ and $\onecount(D) = \rank(D)$, then  $\onecount(D) \geq d$.
        \item If there exists $W \in \Z_p^{m\times d}$ such that the columns of $W$ lie in $\spanZp{[A]_p}$ and $\spanZp{W^T} = \Z_p^d$ and $\rank([D]_p) = \onecount([D]_p)$, then $\onecount([D]_p) \geq d$, which implies $\onecount(D) \geq d)$.
    \end{enumerate}
\end{proof}

\begin{proposition}
\label{prop:EQMATSUBSPAN_P}
    \MatrixSubspanA is in \P.    
\end{proposition}

\begin{proof}
    Recall that \MatrixSubspanA asks the following: given $A \in \Z^{m\times n}$, $d, \ell \in \Z$ where $0 \leq \ell < n$, does there exist $v_1,\dots,v_d \in \Z^n$ such that $Av_i = 0$ for $i \in [1,d]$  and for the matrix $V=(v_1\ \cdots \ v_d)$ we have $\spanZ{(V|_\ell)^T} = \Z^d$?

    We solve \MatrixSubspanA by the following procedure:
    \begin{enumerate}
        \item call the algorithm in \cref{lem:Solve_Ax=b} on input $A \in \Z^{m\times n}$ and $0 = b \in \Z^m$
        \item if this algorithm returns `No', return `No' to \MatrixSubspanA
        \item else let $u_1,\dots,u_m, c \in Z^n$ be the output of the procedure (here $c = 0$), set $U \in \Z^{n\times m}$ be the matrix whose $i$-th column is $u_i$ 
        \item Calculate the SNF $(K,D,L)$ of $U|_\ell$. Thus $K \in \GL(\ell,\Z), L \in \GL(m,\Z), D \in \Z^{\ell\times m}$ and so $U|_\ell = KDL$.
        \item If $\onecount(D) \geq d$ output `Yes', otherwise if $\onecount(D) < d$ output `No'.
    \end{enumerate}

    Step~(1) is polynomial time by \cref{lem:Solve_Ax=b} and step~(2) is polynomial time by \cref{thm:snf_poly} and step~(3) is a straightforward calculation, thus our procedure is polynomial time.
    
    We will now justify the correctness of the procedure.

    If `No' is returned in step~(2), then there does not exist $x \in \Z^n$ which satisfies $Ax=0$, so we output `No' for \MatrixSubspanA. Thus, w.l.o.g we may assume there exists a solution to the procedure in \cref{lem:Solve_Ax=b}, which finds  $U$ and $c$ such that $c = 0$ and $Ax= 0$ if and only if $x \in \spanZ{U}$.
    We then can check if there exists $v_1,\dots,v_d \in \spanZ{U}$ for such that for matrix $V = (v_1 \ \cdots \ v_d)$, we have $\spanZ{(V|_\ell)^T} = \Z^d$, and by \cref{lem:remove_rows_span} $V$ exists if and only if there exists $W \in \Z^{\ell\times d}$ such that $\spanZ{W^T} = \Z^d$ and each column of $W$ lies in $\spanZ{U|_\ell}$.
    Using $D$ of the SNF $(K,D,L)$ calculated in step~(4), by \cref{lem:ifOneCount_then_existv} and the contrapositive of \cref{corr:VT_span_checkOneCount} such a $W$ exists if and only if $\onecount(D) \geq d$, thus justifying the output in step~(5). 
\end{proof}

\subsection{Solving \MatrixSubspanB}
\label{subsec:EQMATBSUBSPANZ}
In this subsection we show that \MatrixSubspanB is decidable in polynomial time. Note that our method in this subsection closely follows \cite[Proposition 4.2 and Lemma 4.3]{FriedlLoh}. 

First, we note the following. 
\begin{lemma}\label{lem:span=Z_iff_gcd1}
    For $a_1,\dots,a_s \in \Z$ and , $\spanZ{\begin{bmatrix}a_1 & \cdots & a_s\end{bmatrix}} = \Z$  if and only if $\gcd(a_1,\dots,a_s) = 1$.
\end{lemma}

It follows that the decision problem 
\MatrixSubspanB can be expressed as follows.\label{pageMATRIXSUBSPANB}
\compproblem{Given $A \in \Z^{m \times n}$, $b \in \Z^m$, $\ell \in \Z$ where $ \ell\in[0,n-1]$.}{Does there exist an integer $n$-matrix $\nu = \begin{bmatrix}\nu_1 &  \cdots & \nu_n\end{bmatrix}^T \in \Z^n$ such that $A\nu + b = 0$ and $\gcd(\nu_{n-\ell+1},\dots,\nu_n) = 1$.}

Next we present some basic facts about $\gcd$, generalising \cite[Lemma 4.3]{FriedlLoh}. 

\begin{lemma}
\label{lem:NEWgcdFacts} Let $n,k,\ell\in \N_+$. 
\be
    \item \label{item:GCD1Equiv} Let $U \in \Z^{\ell\times k}$ with SNF $(K,D,L)$ and $r \in \Z^\ell$. There exists $\mu = (\mu_1 \ \cdots \ \mu_\ell)^T \in \spanX{U}{r} \in \Z^\ell$ such that $\gcd(\mu_1,\dots,\mu_\ell) = 1$ if and only if there exists $v = (v_1 \ \cdots \ v_\ell) \in \spanX{D}{K^{-1}r} \in \Z^\ell$ such that $\gcd(v_1,\dots,v_\ell) = 1$.
    \item\label{itemNEWgcd2} If $s\geq 2$,  $\MFd, b_1,\dots,b_s \in \Z$ such that  $b_i\neq 0$ for at least one $i\in[2,s]$, and $\gcd(\MFd,b_1,\dots,b_s) = 1$,  then there exists $x \in \Z$ such that   $   \gcd(x\MFd + b_1,b_2,\dots,b_s) = 1$.
    \item\label{itemNEWgcd3}  If  $s,\MFd_1,\dots,\MFd_n, b_1,\dots,b_s \in \Z$ and $\MFd_i \mid \MFd_{i+1}$ for $i\in[1,n-1]$, then there exist $a_1,\dots,a_s \in \Z$ such that 
    \begin{equation*}
        \gcd(a_1\MFd_1+b_1, \dots, a_s\MFd_s + b_s) = 1
    \end{equation*}
    if and only if 
    one the following holds:
    \begin{enumerate}
        \item $b_1,\dots,b_s = 0$ and $\MFd_1 \in \{-1,1\}$
        \item $\gcd(b_1,\dots,b_s) = 1$
        \item $\gcd(b_1,\dots,b_s) = c > 1$, $\gcd(\MFd_1, c) = 1$,  and
\be
\item $b_i \not = 0$ for some $i \in [2,s]$ 
   \item  $s=1$ and  $c \equiv 1\mod{\MFd_1}$
   \item $b_2,\dots,b_s = 0$ 
  and $c \equiv 1\mod{\MFd_1}$ or $\MFd_2 \not = 0$.
   \ee

    \end{enumerate}
    \ee
\end{lemma}

\begin{proof} 
    To prove item~(\ref{item:GCD1Equiv}), assume there exists $\mu = (\mu_1 \ \cdots \ \mu_\ell)^T \in \spanX{U}{r}$ with $\gcd(\mu_1,\dots,\mu_\ell) = 1$. Letting $v = K^{-1}\mu$, we have $v \in \spanX{K^{-1}U}{K^{-1}r}$. By \cref{lem:spanA=spanB} $\spanZ{K^{-1}U} = \spanZ{D}$ we have  $v = (v_1 \ \cdots \ v_\ell)^T \in \spanX{D}{K^{-1}r}$. 
    As $v = K^{-1}\mu$ and $K^{-1} \in \GL(\ell,\Z)$,  there exists a sequence of elementary matrices $E_1\cdots E_k = K^{-1}$, and multiplication by an elementary matrix does not change the gcd, so $\gcd(K^{-1}\mu) = \gcd(\mu)$, thus $\gcd(v) = \gcd(\mu) = 1$.

    Item~(\ref{itemNEWgcd2}) is \cite[Lemma 4.3]{FriedlLoh}.

    This leaves item~(\ref{itemNEWgcd3}). Assume one of the conditions (a)--(c) hold. 
    If $b_i=0$ and $\MFd_1=\pm 1$ set $a_1=1$ and $a_j=0$ for $j\geq 2$ and if $\gcd(b_1,\dots, b_s)=1$ then set $a_i=0$.
If $\gcd(b_1,\dots,b_s) = c > 1$, $\gcd(\MFd_1, c) = 1$, we have three subcases.
If $b_i \not = 0$ for some $i \in [2,s]$, 
 then by item~(\ref{itemNEWgcd2}) there exists $x \in \Z$ such that $\gcd(x\MFd_1 + b_1, b_2,\dots,b_s) = 1$ so set $a_1 = x$ and $a_2,\dots,a_s = 0$.
 If $n=1$ then $c=\gcd(b_1)=b_1$, and $c \equiv 1\mod{\MFd_1}$ means $c+\alpha \MFd_1 =1$ for some $\alpha\in \Z$, so setting $a_1=-\alpha$ gives  
$\gcd(a_1\MFd_1+b_1)=a_1\MFd_1+b_1=1$.
If $b_2=\dots=b_s=0$ then $c=\gcd(b_1,0,\dots, 0)=b_1$, so if $c \equiv 1\mod{\MFd_1}$ then again we have  $a_1=\alpha\in \Z$ so that $\alpha b_1+\MFd=1$, and if $\MFd_2\neq 0$
then $\MFd_1\mid \MFd_2$ and $\gcd(\MFd_1, c) = 1$
implies $\gcd(\MFd_1,b_1,\MFd_2) = 1$ with $\MFd_2,b_1\neq 0$ so by item~(\ref{itemNEWgcd2}) there exists $x\in 
Z$ so that $\gcd(x\MFd_1+b_1,\MFd_2)=1$, so set $a_1=x, a_2=1,a_3=\dots=a_s=0$.

Conversely assume there exist $a_1,\dots,a_s \in \Z$ such that
    \begin{equation}\label{eqnGCDmurray}
        \gcd(a_1\MFd_1+b_1, \dots, a_s\MFd_s + b_s) = 1
    \end{equation}

If $b_i=0$ for $i\in[1,s]$, since $\MFd_i \mid \MFd_{i+1}$ for  $i\in[1,s-1]$ we have $\gcd(a_1\MFd_1,\dots, a_s\MFd_s)>|\MFd_1|$ which contradicts \cref{eqnGCDmurray}.
Thus we must have condition (a) or $b_i\neq 0$ for some $i\in[1,s]$.
If $b_i\neq 0$ for some $i\in[1,s]$, either condition (b) holds or $\gcd(b_1,\dots,b_s) = c > 1$.

If 
$f=\gcd(\MFd_1, c) \neq 1$ then $f$ divides every $b_i$ and $\MFd_i$ so no choice of $a_i\in\Z$ can satisfy  \cref{eqnGCDmurray}.
Thus we must have $\gcd(\MFd_1, c)=1$.

Assume condition (c)(i) does not hold. Then either $s=1$ or $s>1$ and $b_i=0$ for $i\in[2,s]$.

If $s=1$ then \cref{eqnGCDmurray} becomes $1=\gcd(a_1\MFd_1+b_1)=a_1\MFd_1+b_1=a_1\MFd_1+c$ since $c=b_1$, and we have condition (c)(ii).

Else $s>1$. If $\MFd_2=0$ then $\MFd_i=0$ for $i>2$, and since $b_i=0$ for $i\in[2,s]$  
\cref{eqnGCDmurray} becomes $1=\gcd(a_1\MFd_1+b_1, 0,\dots, 0)=a_1\MFd_1+b_1$ so $c \equiv 1\mod{\MFd_1}$ and we have condition (c)(iii).
\end{proof}

Recall from p.\pageref{pageMATRIXSUBSPANB} that \MatrixSubspanB may be stated as follows: given $A \in \Z^{m \times n}$, $b \in \Z^m$, $\ell \in \Z$ where $ \ell\in[0,n-1]$, decide if there is an integer $n$-matrix $\nu = (\nu_1, \dots,\nu_n)^T \in \Z^n$ such that $A\nu + b = 0$ and $\gcd(\nu_{n-\ell+1},\dots,\nu_n) = 1$.

\begin{proposition}
\label{prop:EQMATBSUBSPANZ_P}
    \MatrixSubspanB is in \P
\end{proposition}

    \begin{proof}
	We solve \MatrixSubspanB by the following procedure. Given $A \in \Z^{m \times n}$, $b \in \Z^m$, $\ell \in \Z$ where $ \ell\in[0,n-1]$:
\begin{enumerate}
		\item call the algorithm in \cref{lem:Solve_Ax=b} on input $A \in \Z^{m\times n}$ and $b \in \Z^m$
\item if this algorithm returns `No', return `No' to \MatrixSubspanB
\item else let $u_1,\dots,u_m \in \Z^n$ and $c'=(c'_1 \ \cdots \ c'_n)^T\in \Z^n$ be the output of this algorithm, and set $U \in \Z^{n\times m}$ be the matrix whose $i$-th column is $u_i$ 
\item compute the SNF  $(K,D,L)$ of $U|_\ell\in\Z^{\ell\times m}$. Thus $K\in\GL(\ell,\Z)$, $L\in\GL(m,\Z)$, $D\in \Z^{\ell\times m}$ with diagonal entries $\MFd_i$ with $i \in [1,\dots,\rank(D)]$ and $U|_\ell = KDL$. Set:
\begin{itemize}
    \item $\MFd_i = 0$ for $i \in [\rank(D)+1,\ell]$
    \item $c = (c_1 \ \cdots \ c_\ell) = K^{-1}(c'_{n-\ell+1} \ \cdots \ c'_n) \in \Z^\ell$.
\end{itemize}
    \item Return `Yes' if one the following conditions hold, and `No' otherwise: 
		\begin{enumerate}
        \item $c_1,\dots,c_\ell = 0$ and $\MFd_1 \in \{-1,1\}$
        \item $\gcd(c_1,\dots,c_\ell) = 1$
        \item $\gcd(c_1,\dots,c_\ell) = f > 1$, $\gcd(\MFd_1, f) = 1$,  and
    \be
        \item $c_i \not = 0$ for some $i \in [2,\ell]$ 
        \item  $\ell=1$ and  $f \equiv 1\mod{\MFd_1}$
        \item $c_2,\dots,c_\ell = 0$ 
        and $f \equiv 1\mod{\MFd_1}$ or $\MFd_2 \not = 0$.
    \ee
    \end{enumerate}
  \end{enumerate}

	Step~(1) is polynomial time by \cref{lem:Solve_Ax=b}, step~(4) is polynomial time by \cref{thm:snf_poly}, and step~(5) is polynomial time by \cref{lem:poly_processes}. 

    If `No' is returned in step~(2), then there does not exist $x \in \Z^n$ which satisfies $Ax+b=0$, so we output `No' for \MatrixSubspanB. Thus, w.l.o.g we may assume there exists a solution to the procedure in \cref{lem:Solve_Ax=b}, which finds $U$ and $c'$ such that $Ax + b= 0$ if and only if $x \in \spanX{U}{c'}$.
    Let $r = (c'_{n-\ell+1} \ \dots \ c'_n)^T \in \Z^\ell$, we then can check if there exists $\nu \in \spanX{U}{c'}$ such that $\gcd(\nu_1 \ \cdots \ \nu_d) =1$, and by \cref{lem:remove_rows_span} $\nu$ exists if and only if there exists $\mu = (\mu_1 \ \cdots \ \mu_\ell)^T \in \spanX{U|_\ell}{r} \in \Z^{\ell}$ such that $\gcd(\mu_1,\dots,\mu_\ell) = 1$.

    By \cref{lem:NEWgcdFacts}~(\cref{item:GCD1Equiv}) such a $\mu$ exists if and only if there exists $v = (v_1 \ \cdots \ v_\ell)^T \in \spanX{D}{K^{-1}b}$ such that $\gcd(v_1,\dots,v_\ell) = 1$. As $c = K^{-1}r$, all elements in $\spanX{D}{c}$ take the form $a_1\MFd_1 + c_1 + \cdots + a_\ell\MFd_\ell + c_\ell$ for $a_1,\dots,a_\ell \in \Z$, so checking if $v$ exists is equivalent to checking if there exists $a_1,\dots,a_\ell \in \Z$ such that $\gcd(a_1\MFd_1 + c_1,\dots, a_\ell\MFd_\ell + c_\ell) = 1$ which is solved in step~(5) by \cref{lem:NEWgcdFacts}~(\cref{itemNEWgcd3}).
\end{proof}

\section{Concluding remarks}

We have shown the epimorphism problem from the class of finitely presented groups to 
three 
classes of virtually abelian groups is \NP-complete.

The motivation for the present paper was to show that the problems considered by Friedl and L\"oh in 
\cite{FriedlLoh} have \emph{reasonable} 
complexity (by the typical standards of decision problems for finitely presented groups),  
and to extend the classes of virtually abelian and other targets for which the (uniform) epimorphism problem is decidable. We believe that the methods shown in this paper could generalise further, for example to semidirect products $R^d\rtimes_\theta  Q$  where $R$ is a finitely presented principle ideal domain (PID) and $\theta\colon Q\to U(R)$ is an action of a finite group $Q$ on $R^d$ by scalar multiplication, with $U(R)$ the unit group of $R$. 
However, even though we have shown epimorphism is less difficult in those cases than Friedl and L\"oh might have indicated, 
we appear to be no closer to 
resolving their conjecture.

\begin{conjecture}[{\cite[Conjecture 1.5]{FriedlLoh}}]
The uniform epimorphism problem onto the class of all
finitely generated virtually abelian groups is not decidable.
\end{conjecture}

\section*{Acknowledgements}
This research was supported by Australian Research Council grant DP210100271. The second author 
was supported by an Australian Government Research Training Program Stipend.
The third author was partially supported by DFG grant WE 6835/1--2.
The first and third author were 
supported through the program ``Oberwolfach Research Fellows''  at the Mathematisches Forschungsinstitut Oberwolfach in 2023.

We wish to thank  Michal Ferov, Robert Tang, Alexander Thumm and Kane Townsend 
for fruitful discussions about this project, and we are extremely grateful to 
 an anonymous reviewer for their careful reading and detailed and helpful feedback on a previous version of this work.
\bibliographystyle{plain}
\bibliography{epi}

\begin{thebibliography}{1}

\bibitem{Benc2013}
Katalin~A. Bencs{\'a}th, Marianna~C. Bonanome, Margaret~H. Dean, and Marcos
  Zyman.
\newblock Tools: Presentations and their calculus.
\newblock In {\em Lectures on Finitely Generated Solvable Groups}, pages 5--7.
  Springer New York, New York, NY, 2013.

\bibitem{ESW-dihedral}
Murray Elder, Jerry Shen, and Armin Weiß.
\newblock Epimorphism testing with dihedral targets, 2026.
\newblock In preparation.

\bibitem{FriedlLoh}
Stefan Friedl and Clara L\"{o}h.
\newblock Epimorphism testing with virtually {A}belian targets.
\newblock {\em Confluentes Math.}, 13(1):61--78, 2021.

\bibitem{HoltP1989perfect}
D.F. Holt and W.~Plesken.
\newblock {\em Perfect Groups}.
\newblock Oxford mathematical monographs. Clarendon Press, 1989.

\bibitem{kannan1979}
Ravindran Kannan and Achim Bachem.
\newblock Polynomial algorithms for computing the {S}mith and {H}ermite normal
  forms of an integer matrix.
\newblock {\em SIAM J. Comput.}, 8(4):499--507, 1979.

\bibitem{Kuperberg}
Greg Kuperberg and Eric Samperton.
\newblock Computational complexity and 3-manifolds and zombies.
\newblock {\em Geom. Topol.}, 22(6):3623--3670, 2018.

\bibitem{Razborov1985}
A.~A. Razborov.
\newblock Systems of equations in a free group.
\newblock {\em Izv. Akad. Nauk SSSR Ser. Mat.}, 48(4):779--832, 1984.

\bibitem{Remeslennikov}
V.~N. Remeslennikov.
\newblock An algorithmic problem for nilpotent groups and rings.
\newblock {\em Sibirsk. Mat. Zh.}, 20(5):1077--1081, 1167, 1979.

\bibitem{sims1994}
Charles~C. Sims.
\newblock {\em Computation with finitely presented groups}, volume~48 of {\em
  Encyclopedia of Mathematics and its Applications}.
\newblock Cambridge University Press, Cambridge, 1994.

\end{thebibliography}

\end{document}